\documentclass[11pt,amstex]{amsart}%
\usepackage{amsmath,amsthm,amsfonts,amssymb,amscd}
\usepackage[latin1]{inputenc}
\usepackage[all]{xy}
\usepackage[dvips]{graphicx}
\usepackage{amsmath}
\usepackage{amsthm}
\usepackage{amsfonts}
\usepackage{amssymb}%
\providecommand{\U}[1]{\protect\rule{.1in}{.1in}}
\newtheorem{theorem}{Theorem}[section]
\newtheorem{lemma}[theorem]{Lemma}

\newtheorem{proposition}[theorem]{Proposition}
\newtheorem{notation}[theorem]{Notation and convention}
\newtheorem{definition}[theorem]{Definition}

\newtheorem{claim}[theorem]{Claim}
\newtheorem{example}[theorem]{Example}
\newtheorem{conjecture}[theorem]{Conjecture}
\newtheorem{remark}[theorem]{Remark}

\def\U{{\mathcal{U}}}

\def\ov{\overline}

\def\bc{{\mathbb{C}}}

\def\SO(3){\operatorname{{SO(3)}}}
\def\bc{\operatorname{{\mathbb C}}}

\def\ov\bc{\operatorname{\overline{\mathbb{C}}}}

\begin{document}
\title{On holomorphic vector fields with many closed orbits}
\author{Leonardo C\^amara \& Bruno Scardua}
\maketitle

\begin{abstract}
We state some generalizations of a theorem due to G. Darboux, which originally
states that a polynomial vector field in the complex plane exhibits a rational
first integral and has all its orbits algebraic provided that it exhibits
infinitely many algebraic orbits. In this paper, we give an interpretation of
this result in terms of the classical Reeb stability theorems, for compact
leaves of (non-singular) smooth foliations. Then we give versions of Darboux's
theorem, assuring, for a (non-singular) holomorphic foliation of any
codimension, the existence of an open set of compact leaves provided that the
measure of the set of compact leaves is not zero. As for the case of
polynomial vector fields in the complex affine space of dimenion $m\geq2$, we
prove suitable versions of the above results, based also on the very special
geometry of the complex projective space of dimension $m$, and on the nature
of the singularities of such vector fields we consider.

\end{abstract}


\section{Introduction\label{section:intro}}

One of the reasons for studying algebraic solutions of algebraic ordinary
differential equations, is the fact that very basic examples of
\textit{transcendent} (non-algebraic) differentiable functions are given by
solutions of such equations. These solutions however are not common. In fact,
as it is well-known, for a generic choice on the coefficients of the
coordinates, a polynomial complex vector field admits no algebraic solution.
This is the subject of several works (cf. e.g. \cite{LinsNeto},
\cite{LinsNeto-Soares}, \cite{Soares1}, \cite{Soares2}, \cite{Soares3},
\cite{Soares5}, \cite{Soares6}).

On the other hand, an algebraic differential equation in dimension two with a
sufficiently large number of algebraic solutions has all of its solutions in
the algebraic class. This is stated in the following integrability theorem due
to J. G. Darboux:

\begin{theorem}
[Darboux \cite{Jouanolou}]\label{Theorem:Darbouxoriginal} Let $X$ be a
polynomial vector field in ${\mathbb{C}}^{2}$. If $X$ exhibits infinitely many
algebraic orbits, then all orbits are algebraic. In this case $X$ admits a
rational first integral.
\end{theorem}

By an \textit{algebraic orbit} we mean a non-singular orbit which is contained
in an algebraic curve in $\mathbb{C}^{2}$. A rational function
$R(x,y)=P(x,y)/Q(x,y),\,P(x,y),\,Q(x,y)\in\mathbb{C}[x,y]$ is a \textit{first
integral} of $X$ if it is constant along the orbits of $X$. In this case, the
orbits are algebraic and contained in the curves $\lambda P-\mu Q=0$ for all
$(\lambda,\mu)\in\mathbb{C}^{2}\setminus\{0\}$. Thus
Theorem~\ref{Theorem:Darbouxoriginal} is a complete result in dimension two.

Notice that, in the above theorem, the degree of the algebraic leaves has an
upper bound (given by $\operatorname*{max}\{\operatorname{deg}%
(P),\operatorname{deg}(Q)\}$). This somehow motivates the proof which is
outlined as follows. The two basic facts are: (i) The vector field induces a
natural dual polynomial $1$-form $\omega$ of the same degree such that the
integral curves of $\omega$ are the orbits of $X$. (ii) Given a reduced
complex polynomial $f(x,y)$, the equation $\{f(x,y)=0\}$ defines an algebraic
orbit of $X$ if and only if $\frac{1}{f}(\omega\wedge df)=:\theta
_{f}(x,y)dx\wedge dy$ is a polynomial two-form, and in this case we have
$\operatorname{deg}(\theta_{f})=\operatorname{deg}(\omega)-1=(X)-1$. Since the
space of complex polynomials of degree less than $\operatorname{deg}(X)-1$ is
a finite dimensional complex vector space, the hypothesis of infinitely many
algebraic leaves ensures the existence of two linearly independent polynomials
$f(x,y),g(x,y)$ such that $\omega\wedge df=\omega\wedge dg=0$. Therefore
$R=f/g$ is a rational first integral for $X$.

\subsubsection*{Motivation for the foliation-geometrical approach}

The above algebraic argumentation is not available for the case of polynomial
vector fields in dimension $m\geq3$. A first problem involves the definition
of algebraic curve in dimension $m\geq3$ and the fact that not all algebraic
curves are given by complete systems of polynomial equations. Another
complication comes from the fact that it is not clear \textit{a priori} that
the dual space of any vector field $X$\ is generated by $m-1$ polynomial
$1$-forms satisfying the integrability conditions. Thus in order to understand
the higher dimension situation one should deploy other more geometrical
features. This is our main motivation for the introduction of the
\textsl{foliation framework} and also for our current approach.

Theorem~\ref{Theorem:Darbouxoriginal} is a sort of Reeb's complete stability
theorem for codimension one \textsl{projective foliations}, \textit{i.e.},
foliations in complex projective spaces. Let us first recall the classical
result (see for instance \cite{Godbillon, Reeb1}): \textsl{A smooth real
oriented foliation of real codimension one in a compact connected manifold is
a fibration by compact leaves if it exhibits a compact leaf with finite
fundamental group}. This result has many important consequences and motivates
several questions in the theory of foliations. For instance: \textit{Is it
true that a codimension one smooth foliation in a \textrm{(}%
connected\textrm{)} compact manifold with infinitely many compact leaves has
all leaves compact?} The answer is clearly no, but this is true if the
foliation is (transversely) real analytic.

There are versions of Reeb's complete stability theorem for the class of
\emph{holomorphic} foliations (see \cite{Brunella}). In the holomorphic
framework, it is proved in \cite{Brunella-Nicolau} that a
\textit{non-singular} transversely holomorphic \textit{codimension one}
foliation in a compact connected manifold admitting infinitely many compact
leaves exhibits a transversely meromorphic first integral. All foliations
mentioned so far are non-singular foliations. Similarly, in \cite{Ghys} it is
proved that any (possibly singular) \textit{codimension one} holomorphic
foliation in a compact manifold having infinitely many closed leaves admits a
meromorphic first integral. In particular, all leaves are \textsl{closed} off
the singular set.

The problem of bounding the number of closed (off the singular set) leaves of
a holomorphic foliation is known (at least in the complex algebraic framework)
as \textit{Jouanolou's problem}, thanks to the pioneering results in
\cite{Jouanolou}, and has a wide range of contributions and applications in
the algebraic-geometric setting.

From a more analytic-geometrical point of view, in \cite{Santos-Scardua} it is
proved a global stability theorem for codimension $k\geq1$ holomorphic
foliations transverse to fibrations. In \cite{scarduaJMAA} the second author
focus on the problem of existence of a \textit{stable} compact leaf under the
hypothesis of existence of a sufficiently large number of compact leaves.
Recall that a leaf $L$ of a \textit{compact} foliation $\mathcal{F}$ is
\textit{stable\/} if it has a fundamental system of saturated neighborhoods
(cf. \cite{Godbillon}, p. 376). The stability of a compact leaf $L\in
\mathcal{F}$ is equivalent to the finiteness of its holonomy group
$\operatorname{Hol}(\mathcal{F},L)$, and is also equivalent to the existence
of a local bound for the volume of the leaves close to $L$ (\cite{Godbillon},
Proposition 2.20, p. 103). As a partial converse, for smooth codimension one
foliations Reeb proves that a compact leaf admitting a neighborhood consisting
of compact leaves necessarily has finite holonomy. This is not true however in
codimension greater or equal to $2$. We also recall that a subset $X\subset M$
of a differentiable $m$-manifold has \textit{zero measure} if $M$ admits an
open cover by coordinate charts $\varphi:U\subset M\rightarrow{\varphi
}(U)\subset\mathbb{R}^{m}$ such that ${\varphi}(U\cap X)$ has zero measure
(with respect to the standard Lebesgue measure in $\mathbb{R}^{m}$).
Otherwise, we shall say that the set has \textit{positive measure}. In
\cite{scarduaJMAA} it is proved the following measure stability theorem for
non-singular holomorphic foliations: \vglue.1in \noindent\textbf{Theorem A}
(Measure stability theorem) \textsl{A holomorphic foliation $\mathcal{F}$ in a
compact complex manifold $M$ exhibits a compact stable leaf if and only if the
set $\Omega(\mathcal{F})\subset M$ of all compact leaves has positive measure.
} \vglue.1in In this paper we give a more general argumentation for the proof
of Theorem~A, based on the notion of \textit{measure concentration point} of a
subset of positive measure (cf.
Definition~\ref{definition:measureconcentration}). These ideas will be useful
in the remaining part of the paper, in which we deal with the so called
singular case. More precisely, we address the case of one-dimensional
holomorphic foliations in projective spaces. Below we explain why.

\subsubsection*{The singular foliations framework}

The structure of the orbits of a polynomial vector field $X$ in $\mathbb{C}%
^{m}$ is better understood by the study of their asymptotic behavior. For this
sake, we introduce the complex projective space $\mathbb{CP}^{m}$ as a natural
compactification of the affine space $\mathbb{C}^{m}$ and denote by
$E_{\infty}^{m-1}$ the hyperplane at infinity $E_{\infty}^{m-1}=\mathbb{CP}%
^{m}\setminus\mathbb{C}^{m}$. Then we study the structure of the orbits of $X$
in a neighborhood of the hyperplane at infinity. The best way to do that is by
considering the notion of one-dimensional holomorphic foliation with
singularities in complex manifolds. In few words, a one-dimensional
holomorphic foliation with singularities in a complex manifold $M$ consists of
a non-singular one-dimensional holomorphic foliation in a open subset
$U=M\backslash\operatorname*{Sing}(\mathcal{F})$ and a discrete set of points
(called \emph{singularities} of the foliation) denoted by
$\operatorname*{Sing}(\mathcal{F})$. The notions of leaf, holonomy and so on,
refer then to the underlying non-singular foliation and are defined as in the
classical framework. It is well-known that any polynomial vector field defines
a one-dimensional holomorphic foliation with singularities in the
corresponding projective space. Conversely, any projective foliation is
defined in an affine space by a polynomial vector field. Thus our approach
will be based on this correspondence.

By an \textit{algebraic leaf} of such a foliation, we mean a leaf contained in
an algebraic curve. Such leaves are finitely punctured algebraic curves with
such \textquotedblleft ends\textquotedblright\ at singular points of the
foliation. \emph{Algebraic leaves of foliations in projective spaces play, in
a certain sense, the role of compact leaves for (non-singular) smooth
foliations.}

For singular foliations we need to introduce a notion of stability for
algebraic leaves which at first sight seems to be stronger than one should
expected. Let $\mathcal{F}$ be a holomorphic foliation with isolated
singularities in a complex manifold $M$ of dimension $m\geq2$. Given a
non-singular point $q\in M\setminus\operatorname*{Sing}(\mathcal{F})$ and a
transverse disc $\Sigma_{q}$ centered at $q$, we shall denote by $L_{z}$ the
leaf of $\mathcal{F}$ through any $z\in\Sigma_{q}$. The \emph{virtual holonomy
group} of the foliation $\mathcal{F}$ with respect to the transverse section
$\Sigma_{q}$ and base point $q$ is defined as (cf. \cite{C-LN-S2}) the
subgroup $\operatorname*{Hol}\nolimits^{\operatorname*{virt}}(\mathcal{F}%
,\Sigma_{q},q)\subset\operatorname*{Diff}(\Sigma_{q},q)$ of germs of complex
diffeomorphisms that preserve the leaves of the foliation, i.e.,
\[
\operatorname*{Hol}\nolimits^{\operatorname*{virt}}(\mathcal{F},\Sigma
_{q},q)=\{f\in\operatorname*{Diff}(\Sigma_{q},q):\tilde{L}_{z}=\tilde
{L}_{f(z)},\forall z\in(\Sigma_{q},q)\}
\]
Clearly, the virtual holonomy group contains the holonomy group,
\textit{i.e.},
\[
\operatorname*{Hol}(\mathcal{F},L_{q},\Sigma_{q},q)\subset\operatorname*{Hol}%
\nolimits^{\operatorname*{virt}}(\mathcal{F},\Sigma_{q},q),
\]
where $L_{q}\subset M$ is the leaf of $\mathcal{F}$ through $q\in M$. If
$q_{1}$ and $q_{2}$ belong to the same leaf of $\mathcal{F}$, then the
corresponding virtual holonomy groups $\operatorname*{Hol}%
\nolimits^{\operatorname*{virt}}(\mathcal{F},\Sigma_{q_{1}},q_{1})$ and
$\operatorname*{Hol}\nolimits^{\operatorname*{virt}}(\mathcal{F},\Sigma
_{q_{2}},q_{2})$ are holomophically conjugate by a germ of (holonomy)
diffeomorphism $h\colon(\Sigma_{q_{1}},q_{1})\rightarrow(\Sigma_{q_{2}}%
,q_{2})$ for any transverse discs $\Sigma_{1}\ni q_{1}$ and $\Sigma_{2}\ni
q_{2}$. Thus we shall refer to the virtual holonomy group corresponding to a
leaf $L\in\mathcal{F}$ and denote it as $\operatorname*{Hol}%
\nolimits^{\operatorname*{virt}}(\mathcal{F},L)$, for general purposes.

\begin{definition}
[Finite virtual holonomy]\label{Definition:finitevirtualholonomy} \textrm{Let
$\Gamma\subset M$ be a connected invariant subset of }$\mathcal{F}$\textrm{.
We shall say that \textit{each virtual holonomy group of $\Gamma$ is finite}
if for each point $q\in\Gamma\setminus(\Gamma\cap\operatorname*{Sing}%
(\mathcal{F}))$ the virtual holonomy group of the leaf $q\in L_{q}%
\subset\Gamma$ is finite. }
\end{definition}

We stress that $\Gamma$ is a union of leaves and singularities, and \textit{a
priori} the virtual holonomy groups of distinct leaves in $\Gamma$ are not
related to each other. In order to relate such distinct groups we must
introduce the notion of Dulac correspondence associated to suitable situations
involving Siegel singularities (cf. Section~\ref{section:Dulac}).

As mentioned above, algebraic leaves play the role of compact leaves in the
framework of foliations in projective spaces. This idea can be extended in the
sense of the following:

\begin{definition}
[stable and quasi-compact leaf]\label{Definition:stablealgebraicleaf}
\textrm{A leaf of a foliation $\mathcal{F}$ in a complex manifold $M$ will be
called \textit{closed} if it is closed off the singular set
$\operatorname{Sing}(\mathcal{F})$. Therefore, in the one dimensional case the
classical Remmert-Stein extension theorem (\cite{Gunning-Rossi}) ensures that
the closure $\overline{L}\subset L\cup\operatorname{Sing}(\mathcal{F})\subset
M$ is a pure one-dimensional analytic subset of $M$. A closed leaf $L$ is
called \textit{stable }if its virtual holonomy group is finite. By a
\textit{quasi-compact} leaf of $\mathcal{F}$ we shall mean a closed leaf $L$
such that the closure $\overline{L}\subset M$ is compact. This may occur even
if the manifold is not compact, as in the case of product manifolds or fiber
spaces. }
\end{definition}

\begin{remark}
\textrm{If a compact leaf has finite holonomy, then by Reeb's local stability
lemma (Lemma~\ref{Lemma:localstability}) it has a fundamental system of
invariant neighborhoods with all leaves compact. As we shall see
(Section~\ref{section:burnside}, Lemma~\ref{bounded positive measure}), this
implies the finiteness of the leaf's virtual holonomy group. Thus the notion
of stable closed leaf extends the notion of stable compact leaf.}
\end{remark}

The notion of foliation with singularities is detailed in
Section~\ref{section:foliations}. Recall that a \textit{non-degenerate}
singularity of a one-dimensional holomorphic foliation in a complex manifold
is an isolated point in whose neighborhood the foliation is defined by a
holomorphic vector field with a non-singular linear part. Our next result
refers to quasi-compact leaves of foliations with non-degenerate singularities
in surfaces. \vglue.1in \noindent\textbf{Theorem B} \textsl{Let $\mathcal{F}$
be a holomorphic foliation of dimension one with non-degenerate singularities
in the complex surface $M^{2}$. Assume that the set $\Omega(\mathcal{F}%
)\subset M$, union of all quasi-compact leaves of $\mathcal{F}$, has positive
measure. Then $\Omega(\mathcal{F})$ contains some open nonempty set.} \vglue.1in

Indeed, in the above situation we prove the existence of a \textit{stable
graph} (cf. Definition~\ref{Definition:stablegraph}) for the foliation. In
case the surface is compact, this is just the result proved in
\cite{Brunella-Nicolau} for two dimensional complex manifolds.

A non-degenerate singularity is classified as a \textit{Siegel type}
singularity in case the origin $0\in\mathbb{R}^{2}$ belongs to the convex hull
of the set of eigenvalues of the linear part of the vector field at the
singular point. Otherwise, it is classified as a \textit{Poincaré type}
singularity (cf. Section~\ref{section:siegelpoincare}). A Siegel type
singularity in dimension $m$ will be called \textit{normal} if, up to a local
change of coordinates, the coordinate hyperplanes are invariant. In this case,
we may write the vector field as $X=\sum\limits_{j=1}^{m}\lambda_{j}x_{j}%
\frac{\partial}{\partial x_{j}}+(x_{1}\cdots x_{m})\cdot\,X_{2}(x_{1}%
,\cdots,x_{m})$, where $X_{2}$ is holomorphic and vanishes at the origin
$0\in\mathbb{C}^{m}$. As it is well-known, each Siegel type singularity in
dimension $m=2$ is normal (cf. \cite{Dulac},\cite{Mattei-Moussu}). In
dimension $m\geq3$ this occurs if the origin belongs to the interior of the
convex hull of the eigenvalues. In this case the singularity is called a
\textit{strict Siegel} type singularity \cite{Elizarov}.

In terms of foliations in projective spaces, our main result is called
Algebraic Stability theorem and reads as follows: \vglue.1in \noindent
\textbf{Theorem C} (Algebraic Stability theorem) \textsl{Let $\mathcal{F}$ be
a holomorphic foliation of dimension one with non-degenerate singularities in
the $m$-dimensional complex projective space $\mathbb{C}\mathbb{P}^{m}$.
Assume that the set $\Omega(\mathcal{F})\subset\mathbb{C}\mathbb{P}^{m}$,
union of all algebraic leaves of $\mathcal{F}$, has positive measure. Then
$\mathcal{F}$ has a stable algebraic leaf $L_{0}\subset\Omega(\mathcal{F})$.
If the dimension $m=3$ and the \textrm{(}Siegel\textrm{)} singularities are
\textit{normal,} then all singularities in $L_{0}$ are analytically
linearizable exhibiting local holomorphic first integrals.} \vglue.1in We
point-out that in case $m=3$ (and even in case $m\geq3$) we actually prove the
existence of (an algebraic leaf which is contained in) a stable graph (cf.
Definition~\ref{Definition:stablegraphanydimension}). In the situation of
Theorem~C above, we do not know whether there is an invariant nonempty open
subset of $\mathbb{C}\mathbb{P}^{m}$ consisting of algebraic leaves of
$\mathcal{F}$ (see Conjecture \ref{Conjecture:stablegraphproperties}).

\subsubsection*{Outline of the proof of Theorem~C}

A rough sketch of the argumentation in the paper is as follows (we focus on
Theorem~C): As a first step, we prove the existence of a measure concentration
algebraic leaf $L_{0}$, which means an algebraic leaf such that every
neighborhood has a positive measure set of algebraic leaves. Using this and
the fact that the singularities are non-degenerate (and the well known
analytic and topological descriptions of such singularities provided by the
Poincaré linearization theorem and the Poincaré-Dulac normal form), we are
able to prove that this leaf is contained in the interior of the set of
algebraic leaves. This comes from the study of the holonomy groups of the leaf
(via Burnside's theorem) and of the \textit{non-dicritical} (i.e., isolated
from the set of separatrices) adjacent leaves, i.e., leaves accumulating at
singular points contained in the closure of $L_{0}$. An important fact is the
passage, through the leaves, from a neighborhood of $L_{0}$ to a neighborhood
of such a non-dicritical adjacent leaf, which is granted by the construction
of a Dulac map associated to such a singularity corner. This technic allows us
to spread the finiteness properties from the leaf $L_{0}$ to its adjacent
leaves, which are proved to be also algebraic. Thus we construct a kind of
\textquotedblleft stable algebraic graph\textquotedblright\ for the foliation.
This invariant set has finite holonomy groups in a more wide sense, which
shall be introduced later. Also, an important fact is the notion of relative
order of a given leaf, introduced in this paper. The very special geometry of
the complex projective space, as well as the holomorphic character of the
foliation, play a fundamental role in the definition of this notion. Using
classical results from complex geometry, as Chow's theorem on the algebraicity
of analytic sets in complex projective spaces, and Remmert-Stein extension for
analytic subsets of open subsets of complex spaces, we are able to prove that
a leaf of the foliation is algebraic if and only if it has finite relative
order. This characterization and the so called transverse uniformity lemma
allow us to control the behavior of the leaves in the boundary of suitable
sets of algebraic leaves, proving under which conditions these boundary leaves
are also algebraic. In few words, a finite relative order leaf is algebraic,
and a leaf in the boundary of a set of leaves of uniformly bounded order is
also algebraic.

\section{Finite holonomy and local stability}

\label{section:burnside} We begin by recalling one of the very basic tools we
need in this paper. The classical local stability theorem of Reeb reads as
follows (\cite{C-LN,Godbillon}):

\begin{lemma}
[Reeb local stability theorem]\label{Lemma:localstability} Let $L_{0}$ be a
compact leaf with finite holonomy of a smooth foliation $\mathcal{F}$ of real
codimension $k\geq1$ in a manifold $M$. Then there is a fundamental system of
invariant neighborhoods $W$ of $L_{0}$ in $\mathcal{F}$ such that every leaf
$L\subset W$ is compact, has a finite holonomy group and admits a finite
covering onto $L_{0}$. Moreover, for each neighborhood $W$ of $L_{0}$ there is
an $\mathcal{F}$-invariant tubular neighborhood $\pi:W^{\prime}\subset
W\longrightarrow L_{0}$ of $L_{0}$ with the following properties:

\begin{enumerate}
\item Every leaf $L^{\prime}\subset W^{\prime}$ is compact with finite
holonomy group;

\item If $L^{\prime}\subset W^{\prime}$ is a leaf, then the restriction
$\left.  \pi\right\vert _{L^{\prime}}:L^{\prime}\longrightarrow L_{0}$ is a
finite covering map;

\item If $x\in L_{0}$, then $\pi^{-1}(x)$ is a transverse of $\mathcal{F}$;

\item There is an uniform bound $k\in\mathbb{N}$ such that for each leaf
$L^{\prime}\subset W^{\prime}$ we have $\#(L^{\prime}\cap\pi^{-1}(x))\leq k$.
\end{enumerate}
\end{lemma}

Let now $\mathcal{F}$ be a codimension $k$ holomorphic foliation in a complex
manifold $M$. Given a point $p\in M$, the leaf through $p$ is denoted by
$L_{p}$. We denote by $\operatorname*{Hol}(\mathcal{F},L_{p}%
)=\operatorname*{Hol}(L_{p})$ the \textit{holonomy group} of $L_{p}$ (cf.
\cite{C-LN, Godbillon}). This is an equivalence class defined by conjugacy,
and we shall denote by $\operatorname*{Hol}(L_{p},\Sigma_{p},p)$ its
representative evaluated at a local transverse disc $\Sigma_{p}$ centered at
the point $p\in L_{p}$.

The group $\operatorname*{Hol}(L_{p},\Sigma_{p},p)$ is therefore a subgroup of
the group of germs $\operatorname*{Diff}(\Sigma_{p},p)$ which is identified
with the group $\operatorname*{Diff}(\mathbb{C}^{k},0)$ of germs at the origin
$0\in\mathbb{C}^{k}$ of complex diffeomorphisms. \vglue.1in One of the main
tools in the proof of the local stability theorem and in our work is the
following result.

\begin{lemma}
[transverse fibration and transversal uniformity lemma, \cite{C-LN}%
]\label{Lemma:transvuniform} Let $\mathcal{F}$ be a $C^{r}$ foliation in a
manifold $M$. Given a leaf $L\in\mathcal{F}$ and a compact connected subset
$K\subset L$, there are neighborhoods $K\subset U\subset W\subset M$, with $U$
open in $L$ and $W$ open in $M$, and a $C^{r}$ retraction $\pi:W\rightarrow U$
such that the fiber $\pi^{-1}(x)$ is transverse to the restriction $\left.
\mathcal{F}\right\vert _{W}$ for each $x\in U$. Given two points $q_{1}%
,q_{2}\in L_{0}$ in a same leaf $L_{0}$ of $\mathcal{F}$, there are transverse
discs $\Sigma_{1}$ and $\Sigma_{2}$, centered at $q_{1},q_{2}$ respectively,
and a diffeomorphism $h:\Sigma_{1}\rightarrow\Sigma_{2}$ such that for any
leaf $L$ of $\mathcal{F}$ we have $h(L\cap\Sigma_{1})=L\cap\Sigma_{2}$.
\end{lemma}

\section{Periodic groups and groups of finite exponent}

\label{section:generalities}

Next we present Burnside's and Schur's results on periodic linear groups. Let
$G$ be a group with identity $e_{G}\in G$. The group is \textit{periodic} if
each element of $G$ has finite order. A periodic group $G$ is \textit{periodic
of bounded exponent} if there is an uniform upper bound for the orders of its
elements. This is equivalent to the existence of $m\in\mathbb{N}$ with
$g^{m}=1$ for all $g\in G$ (cf. \cite{Santos-Scardua}). Thus a group which is
periodic of bounded exponent is also called a group of \textit{finite
exponent}. The following classical results are due to Burnside and Schur.

\begin{theorem}
[Burnside, 1905 \cite{Burnside}, Schur, 1911 \cite{Schur}]%
\label{Theorem:Burnside} Let $G\subset\operatorname*{GL}(k,\mathbb{C})$ be a
complex linear group.

\begin{enumerate}
\item[\textrm{(i)}] \textrm{(Burnside)} If $G$ is of finite exponent $\ell$
\textrm{(}but not necessarily finitely generated\textrm{)}, then $G$ is
finite; actually we have $\left\vert G\right\vert \leq\ell^{k^{2}}$.

\item[\textrm{(ii)}] \textrm{(Schur)} If $G$ is finitely generated and
periodic \textrm{(}not necessarily of bounded exponent\textrm{)}, then $G$ is finite.
\end{enumerate}
\end{theorem}

\noindent Using this result we prove the following

\begin{lemma}
[\cite{Santos-Scardua}]\label{Lemma:finiteexponentgerms} About periodic groups
of germs of complex diffeomorphisms we have:

\begin{enumerate}
\item A finitely generated periodic subgroup $G\subset\operatorname*{Diff}%
(\mathbb{C}^{k},0)$ is necessarily finite. A \textrm{(}non necessarily
finitely generated\textrm{)} subgroup $G\subset\operatorname*{Diff}%
(\mathbb{C}^{k},0)$ of finite exponent is necessarily finite.

\item Let $G\subset\operatorname*{Diff}(\mathbb{C}^{k},0)$ be a finitely
generated subgroup. Assume that there is an invariant connected neighborhood
$W$ of the origin in $\mathbb{C}^{k}$ such that each point $x$ is periodic for
each element $g\in G$. Then $G$ is a finite group.

\item Let $G\subset\operatorname*{Diff}(\mathbb{C}^{k},0)$ be a \textrm{(}non
necessarily finitely generated\textrm{)} subgroup such that for each point $x$
close enough to the origin, the pseudo-orbit of $x$ is periodic of
\textrm{(}uniformly bounded\textrm{)} order $\leq\ell$ for some $\ell
\in{\mathbb{N}}$, then $G$ is finite.
\end{enumerate}
\end{lemma}

The above lemma is partially extended in
Lemma~\ref{Lemma:concentrationboundary}.

\section{Measure and finiteness}

Let us pave the way to the proof of Theorem~A. For the sake of simplicity, we
will adopt the following notation: if a subset $X\subset M$ is not a zero
measure subset, then we shall say that it has \textit{positive measure} and
write $\operatorname*{med}(X)>0$. This may cause no confusion since we are not
considering any specific measure on $M$ and we shall be dealing only with the
notion of zero measure subset stated in Section~\ref{section:intro}.
Nevertheless, we notice that if $X\subset M$ writes as a countable union
$X=\bigcup\nolimits_{n\in\mathbb{N}}X_{n}$ of subsets $X_{n}\subset M$, then
$X$ has zero measure in $M$ if and only if $X_{n}$ has zero measure in $M$ for
\textit{all} $n\in\mathbb{N}$. In terms of our notation, we have therefore
$\operatorname*{med}(X)>0$ if and only if $\operatorname*{med}(X_{n})>0$ for
\textit{some} $n\in\mathbb{N}$.

\subsection{Measure concentration points}

Here we introduce a fundamental notion in our argumentation.

\begin{definition}
\label{definition:measureconcentration} \textrm{Given a subset $X\subset M$ of
a differentiable manifold, a point $p\in M$ will be called a \textit{measure
concentration point} of $X$ if the set $V\cap X$ has positive measure in $M$
for any open neighborhood $p\in V\subset M$. The set of measure concentration
points of $X$ will be denoted by ${\mathcal{C}}$$_{\mu}(X)$. Clearly,
${\mathcal{C}}$$_{\mu}(X)\subset\overline{X}$. If we denote by
$\operatorname*{Int}$$(X)\subset M$ the set of interior points of $X$, then
$\operatorname*{Int}(X)\subset\mathcal{C}_{\mu}(X)$. }
\end{definition}

\begin{lemma}
\label{Lemma:concentrationboundary} If $X\subset M$ is a subset, then a
boundary point $p\in\partial{\mathcal{C}}_{\mu}(X)$ is a measure concentration
point of $X$. In other words, $\partial{\mathcal{C}}_{\mu}(X)\subset
{\mathcal{C}}_{\mu}(X)$.
\end{lemma}

\begin{proof}
Without loss of generality, we may assume $M$ is an open subset of
$\mathbb{R}^{n}$. Given a neighborhood $V\ni p$, there is a small $R>0$ such
that the ball centered at $p$ with radius $R\ $ is contained in $V$, i.e.,
ball $B(p;R)\subset V\subset\mathbb{R}^{n}$. Since $p\in\partial
\mathcal{C}_{\mu}(X)$, there is a point $q\in{\mathcal{C}}_{\mu}(X)$ such that
$|q-p|<R/2$. Given now the neighborhood $B(q;R/3)$ of $q$, since
$q\in{\mathcal{C}}_{\mu}(X)$, we have $\operatorname*{med}(X\cap B(q;R/3))>0$.
On the other hand, we have $X\cap V\supset X\cap B(p;R)\supset X\cap
B(q;R/3)$. Therefore, $\operatorname*{med}(X\cap V)>0$. This proves that
$p\in{\mathcal{C}}_{\mu}(X)$.
\end{proof}

\begin{lemma}
\label{Lemma:nondiscrete} Given a subset $X\subset M$, the set $\mathcal{C}%
_{\mu}(X)$ has no isolated points.
\end{lemma}

\begin{proof}
Suppose that $p\in\mathcal{C}_{\mu}(X)$ is an isolated measure concentration
point, then there is an open neighborhood $W$ of $p$ in $M$ such that
$(W\setminus\{p\})\cap\mathcal{C}_{\mu}(X)=\emptyset$. Therefore, given a
point $q\in W\setminus\{p\}$, there is an open neighborhood $q\in V_{q}\subset
W\setminus\{p\}$ of $q$ such that $X\cap W_{q}$ has zero measure. The open
cover $W\setminus\{p\}\subset\bigcup_{q\in W\setminus\{p\}}W_{q}$ admits a
countable subcover $W\setminus\{p\}\subset\bigcup_{j\in\mathbb{N}}W_{q_{j}}$.
Since $X\cap W_{q_{j}}$ has zero measure for each $j\in\mathbb{N}$, we
conclude that $X\cap W$ has zero measure, therefore $p$ cannot be a measure
concentration point of $X$. This leads to a contradiction.
\end{proof}

\begin{lemma}
\label{Lemma:concentrationgraduation} Let $X\subset M$ be a subset such that
$X=\bigcup\nolimits_{k\in\mathbb{N}}X_{k}$, where $\mathcal{C}_{\mu}%
(X_{k})=\emptyset$ for each $k\in\mathbb{N}$. Then $\operatorname*{med}(X)=0$.
\end{lemma}

\begin{proof}
Since $\mathcal{C}_{\mu}(X_{k})=\emptyset$, it follows that
$\operatorname*{med}(X_{k})=0$ for each $k\in\mathbb{N}$. Thus
$\operatorname*{med}(X)=0$.
\end{proof}

Indeed, we can prove more:

\begin{lemma}
\label{Lemma:concentrationinside} Let $X\subset M$ be a subset such that
$X\cap\mathcal{C}_{\mu}(X)=\emptyset$. Then $\operatorname*{med}(X)=0$.
\end{lemma}

\begin{proof}
Since $X\cap\mathcal{C}_{\mu}(X)=\emptyset$, for each point $x\in X$ there is
a neighborhood $x\in W_{x}\subset M$ such that $\operatorname*{med}(W_{x}\cap
X)=0$. Therefore $X\subset\bigcup\nolimits_{{x\in X}}W_{x}$. By choosing a
countable subcover $X\subset\bigcup\nolimits_{{j\in\mathbb{N}}}W_{x_{j}}$ we
obtain $X\subset\bigcup\nolimits_{{j\in\mathbb{N}}}(W_{x_{j}}\cap X)$ and
since $\operatorname*{med}(X\cap W_{x_{j}})=0,\forall j\in\mathbb{N}$,
conclude that $\operatorname*{med}(X)=0$.
\end{proof}

The following result is very natural.

\begin{lemma}
\label{Lemma:invariant} Let $\Omega\subset M$ be an $\operatorname{\mathcal
{F}}$-invariant subset. Then the set $\mathcal{C}_{\mu}(\Omega)$ is also
invariant by the foliation $\mathcal{F}$.
\end{lemma}

\begin{proof}
Consider a non-singular point $q\in\mathcal{C}_{\mu}(\Omega)$, using a local
trivialization chart $(W,\varphi)$ for $\mathcal{F}$, we conclude that the
plaque $q\in Q\subset W$ of the leaf $L_{q}$ of $\mathcal{F}$ through $q$ is
also contained in the set $\mathcal{C}_{\mu}(\Omega)$. This shows that
$\mathcal{C}_{\mu}(\Omega)$ is (locally invariant and therefore) invariant.
\end{proof}

\subsection{Subgroups with uniformly bounded pseudo-orbits orders}

\label{subsection:uniformlybounded} Now we shall obtain a slight
generalization of Lemma \ref{Lemma:finiteexponentgerms}.(3). First we need
some notation. Let $S$ be a complex manifold, $S_{p}:=(S,p)$ the germ of $S$
at the point $p\in S$, and $V$ a neighborhood of $p$ in $S$. Then we say that
$q\in V$ is a \emph{periodic point} with respect to $G$ (or $G$%
\emph{-periodic} point for short) if any germ $g\in G$ has a representative
map $g:V\longrightarrow S$ such that $g^{\circ(j)}(q)\in V$ for all
$j=0,1,\ldots,k_{g}$ and $g^{\circ(k_{g})}(q)=q$. In particular, the minimum
possible value $k_{g}\in\mathbb{N}$ satisfying the previous property is called
the \emph{order} of the pseudo-orbit of $x$. Further, we say that $q$ is a
periodic point with respect to $G$ with \emph{uniformly bounded pseudo-orbits}
if there is $\ell\in\mathbb{N}$ such that $k_{g}\leq\ell$ for all $g\in G$.
For any subgroup $G\subset\operatorname*{Diff}(S_{p},p)$ we denote by
$\operatorname*{O}_{G}(S_{p},\ell)$ the set of $G$-periodic points whose
pseudo-orbits orders are uniformly bounded by $\ell\in\mathbb{N}$.

\begin{lemma}
\label{bounded positive measure} Let $G$ be a \textrm{(}non necessarily
finitely generated\textrm{)} subgroup of $\operatorname*{Diff}(S_{p},p)$ such
that $p\in\mathcal{C}_{\mu}(\operatorname*{O}_{G}(S_{p},\ell))$ for some
$\ell\in{\mathbb{N}}$. Then $G$ is finite.
\end{lemma}

\begin{proof}
Now consider a map germ $g\in G$ and pick a neighborhood $V$ of $p$ in $S$
such that the representative $g:V\rightarrow S$ (of the germ $g$) and its
iterates $g,g^{2},\cdots,g^{\ell}$ are defined in $V$. Since the orders of the
orbits in $\operatorname*{O}_{G}(S_{p},\ell)$ are uniformly bounded by $\ell$,
then $\operatorname*{O}_{G}(S_{p},\ell)\cap V$ is contained in the analytic
subset $X_{\ell}:=\bigcup_{m=0}^{\ell}\{z\in V:g^{\circ(m)}(z)=z\}$.
Therefore, since $\operatorname*{med}(\operatorname*{O}_{G}(S_{p},\ell))>0$
for some $\ell\in\mathbb{N}$, then $\operatorname*{O}_{G}(S_{p},\ell)\cap V=V$
(i.e., $g^{\circ(k)}=\operatorname*{id}$, for some $k\leq\ell$). This shows
that each germ $g\in G$ is periodic of order $k_{g}\leq\ell$ for some uniform
$\ell\in\mathbb{N}$. This implies that $G$ is finite by
Lemma~\ref{Lemma:finiteexponentgerms}.
\end{proof}

\subsection{Proof of the measure stability theorem}

Let us sketch the proof of Theorem~A. First we show the following preliminary
results for a compact complex manifold $M$ endowed with a non-singular
holomorphic foliation $\mathcal{F}$.

\begin{claim}
\label{Claim:completetransversal} There is a finite number of relatively
compact open discs $T_{j}\subset M$, $j=1,\ldots,r$, such that:

\begin{enumerate}
\item Each $T_{j}$ is transverse to $\mathcal{F}$ and the closure
${\overline{T_{j}}}$ is contained in the interior of a transverse disc
$\Sigma_{j}$ to $\mathcal{F}$;

\item Each leaf of $\mathcal{F}$ intersects at least one of the discs $T_{j}$.
\end{enumerate}
\end{claim}

\begin{proof}
Since $M$ is compact, it is enough to show that, for each point $p\in M$,
there is an open neighborhood $U_{p}\subset M$ of $p$, and a relatively
compact open disc $T_{p}\subset U_{p}$ whose closure $\overline{T_{p}}$ is
contained in the interior of a disc $\Sigma_{p}$ transverse to $\mathcal{F}$,
and such that each leaf of $\mathcal{F}$ intersecting $U_{p}$ crosses the disc
$T_{p}$. But this is an immediate consequence of the local trivialization
property of foliations and of the fact that $M$ is a locally compact
topological space.
\end{proof}

Let $\{T_{1},\cdots T_{r}\}\subset M$ be as in the above claim, then we call
$T=\bigcup_{j=1}^{r}T_{j}$ a \textit{complete transversal} to $\mathcal{F}$ in
$M$.

\begin{remark}
\textrm{As it is well known, the set of leaves of a foliation is not
necessarily a manifold (not even a Hausdorff topological space). Therefore,
introducing the concept of a complete transversal to $\mathcal{F}$, the above
claim allows to use the notion of measure concentration point in the space of
leaves of a foliation $\mathcal{F}$ defined in a compact manifold $M$.}
\end{remark}

Let%
\[
\Omega(\mathcal{F},T)=\{L\in\mathcal{F}:\#(L\cap T)<\infty\}\text{,}%
\]
then a leaf $L\in\Omega(\mathcal{F},T)$ is called a \textit{finite order leaf}
with respect to the complete transversal $T$. In particular, $\Omega
(\mathcal{F},T)=\bigcup_{k=1}^{\infty}\Omega(\mathcal{F},T,k)$, where
\[
\Omega(\mathcal{F},T,k)=\{L\in\mathcal{F}:\#(L\cap T)\leq k\}.
\]
Since $M$ is compact we have:

\begin{claim}
\label{Claim:finiteordercompact} A leaf of $\mathcal{F}$ is compact if and
only if it has finite order. In other words, $\Omega(\mathcal{F}%
)=\Omega(\mathcal{F},T)$ as collections of leaves\footnote{Although
$\Omega(\mathcal{F})$ and $\Omega(\mathcal{F},T)$ are considered as
collections of leaves, when we refer to the measure of these sets, we are
considering the measure of the union of the leaves in each of these sets. This
should cause no misunderstanding.}.
\end{claim}

\begin{proof}
A compact leaf intersects a complete transversal only a finite number of
times, since the union of the intersection points is a discrete subset of a
compact set. Conversely, assume a leaf $L\in\mathcal{F}$ has finite order with
respect to a complete transversal $T$ as above. We claim that $L\subset M$ is
closed and therefore compact. In fact, if $L$ is not closed, then it has an
accumulation point $p\in\overline{L}\setminus L$. Using the local
trivializations of $\mathcal{F}$ at $p$, we conclude that there is an
arbitrarily small transverse disc $\Sigma_{p}$, centered at $p$, such that
$L\cap\Sigma_{p}=+\infty$. The leaf $L_{p}\ni p$ necessarily intersects the
collection $T$ at some interior point, say $q\in L_{p}\cap T_{j}$, for a
certain $j\in\{1,\ldots,r\}$. Choose a simple path $\gamma\colon
\lbrack0,1]\rightarrow L_{p}$ joining $\gamma(0)=p$ to $\gamma(1)=q$. Then the
corresponding holonomy map gives a germ of diffeomorphism $h_{\gamma}%
\colon(\Sigma_{p},p)\rightarrow(T_{j},q)$. Since for any arbitrarily small
disc $\Sigma_{p}$ we have $\#(L\cap\Sigma_{p})=+\infty$, choosing a
representative for the above holonomy map, we have $\#(h_{\gamma}(L\cap
\Sigma_{p})\cap T_{j})=+\infty$. Since $h_{\gamma}(L\cap\Sigma_{p})\subset
L\cap T_{j}$, we conclude that $\#(L\cap T_{j})=+\infty$. Therefore $L$ cannot
have finite order with respect to $T$. This proves the claim.
\end{proof}

\begin{remark}
\label{Remark:finiteorder} \textrm{It is important to notice that even if
\textit{a priori} the notion of order cannot be defined with respect to a
given complete transversal, thanks to the above claim, the notion of finite
order is intrinsic to the leaf. }
\end{remark}

Now let us deal with the proof of the theorem itself. Since the necessary part
of the statement is immediate, we only have to verify the sufficient part.
Thus, from now on, we assume $\operatorname*{med}(\Omega(\mathcal{F}))>0$. But
recall from Claim \ref{Claim:finiteordercompact} that $\Omega(\mathcal{F}%
)=\Omega(\mathcal{F},T)=\bigcup_{n\in\mathbb{N}}\Omega(\mathcal{F},T,n)$, thus
there is $n\in\mathbb{N}$ such that $\operatorname*{med}(\Omega(\mathcal{F}%
,T,n))>0$.\vglue.1in

Now recall that a leaf $L\in\Omega(\mathcal{F})$ is a \textit{measure
concentration point} of the set of compact leaves $\Omega(\mathcal{F})$ if for
any open neighborhood $W$ of $L$ the intersection $W\cap\Omega(\mathcal{F})$
has positive measure.

\begin{claim}
\label{Claim:existsleaf}There is a compact leaf $L_{0}\subset\mathcal{C}_{\mu
}(\Omega(\mathcal{F}))$.
\end{claim}

\begin{proof}
Suppose that for each compact leaf $L\in\Omega(\mathcal{F})$ and each
neighborhood $V_{L}$ of $L$ in $M$ there is a neighborhood $W_{L}\subset
V_{L}$ of $L$ in $M$ such that $\operatorname*{med}(W_{L}\cap\Omega
(\mathcal{F}))=0$. In particular, there is an open cover $\Omega
(\mathcal{F})\subset\bigcup_{L\in\Omega(\mathcal{F})}W_{L}$ such that
$\operatorname*{med}(W_{L}\cap\Omega(\mathcal{F},T,n))=0$. Since this open
cover admits a countable subcover $\Omega(\mathcal{F})\subset\bigcup
_{n\in\mathbb{N}}W_{n}$ with $\operatorname*{med}(W_{n})=0$ for all
$n\in\mathbb{N}$, then $\operatorname*{med}(\Omega(\mathcal{F}))=0$; a contradiction.
\end{proof}

Applying now Claim~\ref{Claim:finiteordercompact} and
Lemma~\ref{Lemma:concentrationgraduation} we conclude that $L_{0}%
\subset\mathcal{C}_{\mu}(\Omega(\mathcal{F},T,n))$ for some $n\in\mathbb{N}$.

\begin{claim}
\label{Claim:holonomyisfinite} The holonomy group of $L_{0}$ is finite.
\end{claim}

\begin{proof}
Since $\operatorname*{med}(\Omega(\mathcal{F}))>0$, Claim
\ref{Claim:finiteordercompact} and the comments before it guarantee the
existence of a positive integer $n\in\mathbb{Z}_{+}$ such that
$\operatorname*{med}(\Omega(\mathcal{F},T,n))>0$. Now pick $p\in L_{0}\cap T$
and a disc $\Sigma\subset{\overline{\Sigma}}\subset T$ transverse to
$\mathcal{F}$ and centered at $p$. For each $z\in\Sigma$, denote the leaf
through $z$ by $L_{z}$. If $L_{z}\in\Omega(\mathcal{F},T,n)$, then
$\#(L_{z}\cap\Sigma)\leq n$. Let $X:=\{z\in W:\#(L_{z}\cap\Sigma)\leq n\}$.
Since $\operatorname*{med}(\Omega(\mathcal{F},T,n))>0$, then Claim
\ref{Claim:existsleaf} ensures that $\operatorname*{med}X>0$.

Now consider a holonomy map germ $h\in H:=\operatorname*{Hol}(\mathcal{F}%
,L_{0},\Sigma,p)$ and choose a sufficiently small subdisc $W\subset\Sigma$
such that the representative $h:W\rightarrow\Sigma$ (of the germ $h$) and its
iterates $h,h^{2},\cdots,h^{n+1}$ are defined in $W$. Since $X$ has positive
measure, then $\operatorname*{med}(\operatorname*{O}_{H}(\Sigma,n))>0$. The
result then follows by Lemma \ref{bounded positive measure}.
\end{proof}

In view of Claim~\ref{Claim:holonomyisfinite} and Reeb local stability
Lemma~\ref{Lemma:localstability}, the proof of Theorem~A is finished.

\section{Holomorphic foliations with singularities}

\label{section:foliations}

Recall that a \textit{singular holomorphic foliation} in a complex manifold
$M$ of dimension $m\geq2$ is a pair $\mathcal{F}=(\mathcal{F}^{\prime
},\operatorname*{Sing}(\mathcal{F}))$, where $\operatorname*{Sing}%
(\mathcal{F})\subset M$ is an analytic subset of $M$ of dimension less or
equal to $(\mathcal{F})-1$, and $\mathcal{F}^{\prime}$ is a holomorphic
foliation in the usual sense (without singularities) in the open set
$M^{\prime}=M\setminus\operatorname*{Sing}(\mathcal{F})\subset M$. The
\textit{leaves} of $\mathcal{F}$ are defined as the leaves of the foliation
$\mathcal{F}^{\prime}$. The set $\operatorname*{Sing}(\mathcal{F})$ is called
the \textit{singular set} of $\mathcal{F}$. The \textit{dimension} of
$\mathcal{F}$ is defined as the dimension of $\mathcal{F}^{\prime}$. In the
one-dimensional case there is an open cover $\{U_{j}\}_{j\in J}$ of $M$ such
that in each $U_{j}$ the foliation $\mathcal{F}$ is defined by a holomorphic
vector field $X_{j}$ satisfying the following property: if $U_{i}\cap
U_{j}\neq\emptyset$, then $\left.  X_{i}\right\vert _{U_{i}\cap U_{j}}%
=g_{ij}\cdot\,\left.  X_{j}\right\vert _{U_{i}\cap U_{j}}$ for some
non-vanishing holomorphic function $g_{ij}$ in $U_{i}\cap U_{j}$. The leaves
of the restriction $\left.  \mathcal{F}\right\vert _{U_{j}}$ are the
nonsingular orbits of $X_{j}$ in $U_{j}$ while $\operatorname*{Sing}%
(\mathcal{F})\cap U_{j}=\operatorname*{Sing}(X_{j})$. An isolated singularity
of a one-dimensional foliation $\mathcal{F}$ in a manifold $M$ is called
\textit{non-degenerate} if there is some open neighborhood $p\in U\subset M$
where the foliation is induced by a holomorphic vector field $X$ with
non-singular linear part $DX(p)$ at $p\in\operatorname*{Sing}(\mathcal{F})$.

Let $p\in M$ be an isolated singularity of a one dimensional singular
foliation $\mathcal{F}$ in $M$. Given a neighborhood $p\in U\subset M$ where
$\mathcal{F}$ has no other singularity than $p$, we denote by $\mathcal{F}%
(U):=\left.  \mathcal{F}\right\vert _{U}$ the restriction of $\mathcal{F}$ to
$U$. A leaf of $\mathcal{F}(U)$ accumulating only at $p$ is closed off $p$,
thus by Remmert-Stein extension theorem (\cite{Gunning-Rossi}) it is contained
in an irreducible analytic curve through $p$. Such a curve is called a local
\emph{separatrix} of $\mathcal{F}$ through $p$. In the two dimensional case,
the classical definition says that $p\in\operatorname*{Sing}(\mathcal{F})$ is
a \emph{dicritical singularity} if for some neighborhood $p\in U\subset M$ the
restriction $\mathcal{F}(U)$ has infinitely many separatrices through $p$. We
extend this terminology in a natural way to higher dimensions: for
$m=(\mathcal{F})\geq2$, the singularity shall be called \emph{dicritical} if
it admits infinitely many separatrices. This singularity is called
\emph{absolutely dicritical} if all the leaves close enough to it are
contained in local separatrices. For our purposes, the important factor is
whether the set of separatrices has positive measure or not. In dimension
$m\geq2$, we shall say that $p$ is a $\mu$\emph{-dicritical} singularity if
for arbitrarily small open neighborhoods $U\ni p$ the set $\operatorname*{Sep}%
(\mathcal{F},U)$ of local separatrices of $\mathcal{F}(U)$ through $p$ has
positive measure in $U$.


In dimension two, an isolated singularity is dicritical if and only if it is
$\mu$-dicritical. This is a straightforward consequence of the theorem of
resolution of singularities by blow-ups (\cite{C-LN-S2, seidenberg}).
Nevertheless, in dimension three a linear vector field with eigenvalues $1$,
$1$ and $-1$ is not $\mu$-dicritical, but exhibits infinitely many
separatrices (contained in the plane spanned by the positive eigenvalues).

By Newton-Puiseaux parametrization theorem, the topology of a separatrix is
the one of a disc. Further, the separatrix minus the singularity is
biholomorphic to a punctured disc. In particular, given a separatrix $S_{p}$
through a singularity $p\in\operatorname*{Sing}(\mathcal{F})$, we may choose a
loop $\gamma\in S_{p}\setminus\{p\}$ generating the (local) fundamental group
$\pi_{1}(S_{p}\setminus\{p\})$. The corresponding holonomy map $h_{\gamma}$ is
defined in terms of a germ of complex diffeomorphism at the origin of a local
disc $\Sigma$ transverse to $\mathcal{F}$ and centered at a non-singular point
$q\in S_{p}\setminus\{p\}$. This map is well-defined up to conjugacy by germs
of holomorphic diffeomorphisms, and is generically referred to as
\textit{local holonomy} of the separatrix $S_{p}$ with respect to the
singularity $p$. Let us denote by $\Omega(\mathcal{F},\Sigma,k)$ the union of
the leaves $L$ of $\mathcal{F}$ such that $L$ meets the transverse disc
$\Sigma$ at most in $k$ points, \textit{i.e.}, $\#(L\cap\Sigma)\leq k$.

\begin{lemma}
\label{Lemma:finitelocalholonomy} Let $\mathcal{F}$ be a one-dimensional
holomorphic foliation with an isolated singularity at $p\in M$. Let $S$ be a
local separatrix of $\mathcal{F}$ through $p$ and $\Sigma$ a local disc
transversal to $\mathcal{F}$ centered at a non-singular point $q\in
S\setminus\{p\}$ (close enough to $p$) such that $S\subset{\mathcal{C}}_{\mu
}(\Omega(\mathcal{F},\Sigma,k))$ for some $k\in\mathbb{N}$. Then the local
holonomy of $S$ with respect to $p$ is a periodic map.
\end{lemma}

\begin{proof}
Let $H:=\operatorname*{Hol}(\mathcal{F},L,\Sigma,q)$. Since $S\subset
{\mathcal{C}}_{\mu}(\Omega(\mathcal{F},\Sigma,k))$, then $\operatorname*{med}%
(\operatorname*{O}_{H}(\Sigma_{q},k))>0$. The result then follows by Lemma
\ref{bounded positive measure}.
\end{proof}

\section{Siegel type and Poincaré type singularities}

\label{section:siegelpoincare}

Let $\mathcal{F}$ be a one-dimensional holomorphic foliation with
singularities in a manifold $M$. Assume $p\in\operatorname*{Sing}%
(\mathcal{F})$ is a \textit{nondegenerate singularity}, \textit{i.e.}, for
some neighborhood $p\in U\subset M$ the restriction $\left.  \mathcal{F}%
\right\vert _{U}$ is given by a holomorphic vector field $X$ with a
non-singular linear part at $p$. There are two possibilities: If the convex
hull in $\mathbb{R}^{2}$ of the set of eigenvalues of the linear part $DX(p)$
contains the origin, then we say that the singularity is in the \textit{Siegel
domain}, otherwise it is the \textit{Poincaré domain} (\cite{Arnold, Dulac,
CKP, IlyYak2006}).

\subsection{Singularities in the Poincaré domain\label{subsubsection:poincare}%
}

A singularity in the Poincaré domain is either analytically linearizable or
exhibits resonant eigenvalues and is analytically conjugate to a polynomial
form with resonant monomials called Poincaré-Dulac normal form (\cite{Arnold,
Dulac}). For such a non-linear normal form, the separatrices are contained in
coordinate hyperplanes and not all hyperplanes are invariant. In particular,
the set of separatrices is a zero measure subset.

An immediate consequence of the above discussion and of Poincaré-Dulac normal
form theorem (\cite{Dulac, Arnold}) for singularities in the Poincaré domain
is the following result.

\begin{lemma}
\label{Lemma:poincareholonomylinearizable} A $\mu$-dicritical non-degenerate
singularity is necessarily in the Poincaré domain and is analytically
linearizable: in suitable local coordinates, it is of the form $\dot{x}=Ax$
for some diagonal linear map $A\in\operatorname*{GL}(m,\mathbb{Z}_{+})$ with
positive integer coefficients. The same holds if the set of leaves which are
closed off the singular set is a positive measure set.
\end{lemma}

A singularity as above will be called \textit{radial type singularity}.

\subsection{Singularities in the Siegel domain\label{subsubsection:siegel}}

We consider the following situation: $\mathcal{F}$ is a one-dimensional
foliation defined in a neighborhood of the origin $0\in\mathbb{C}^{m}$, which
is assumed to be a (non-degenerate) Siegel type singularity. We recall that
the singularity is of normal type if we can choose small polydiscs $U\ni0$
centered at the origin endowed with local coordinates $x=(x_{1},\cdots
,x_{m})\in U$ such that:

\begin{enumerate}
\item[(i)] The coordinate hyperplanes $H_{j}=\{x_{j}=0\}$ are invariant by the foliation.

\item[(ii)] The coordinate axes $\mathcal{O}_{x_{j}}$ contain separatrices
whose holonomy maps, denoted by $h_{j}$, are related to the analytic
classification of the germ of the foliation at the origin (see
section~\ref{section:elizarov}).
\end{enumerate}

Coordinates as above will be called \textit{adapted} to the singularity. The
existence of such adapted coordinates (in dimension $m\geq3$) for any Siegel
type singularity is discussed in \cite{Camara-Scardua-separatrix}. Fix a
holomorphic vector field $X$ in $U$ with an isolated singularity at $p$
defining the restriction $\mathcal{F}(U):=\left.  \mathcal{F}\right\vert _{U}%
$. Let $\lambda_{j}\in\mathbb{C}$ be the eigenvalue of $DX(0)$ corresponding
to the eigenvector tangent to the $\mathcal{O}_{x_{j}}$ axis. For a suitable
choice of $X$ and of the local \textit{adapted} coordinates we can write%
\begin{equation}
X=\sum\limits_{j=1}^{m}\lambda_{j}x_{j}\frac{\partial}{\partial x_{j}}%
+(x_{1}\cdots x_{m})\cdot\,X_{2}(x_{1},\cdots,x_{m})
\label{equation:Siegelvectorfield}%
\end{equation}
where $\lambda_{j}\in\mathbb{C},\forall j$, and $X_{2}$ is a vector field
defined in a neighborhood of the origin of the coordinate system
$(x_{1},\cdots,x_{m})$.

As we shall see, under suitable conditions we can assure that the eigenvalues
are integral numbers (cf. Lemma~\ref{bounded positive measure}). Indeed, we
may choose transverse discs $\Sigma_{j}$ centered at points $q_{j}%
:(x_{j}=a_{j})\in\mathcal{O}_{x_{j}}$. Then the holonomy maps $h_{j}$ have
representatives given by local diffeomorphisms $h_{j}:(\Sigma_{j}%
,q_{j})\rightarrow(\Sigma_{j},q_{j})$ of the form
\[
h_{j}(x_{1},\cdots,\widehat{x_{j}},\cdots,x_{n})=(x_{1}a_{1}^{j}(x_{1}%
,\cdots,\widehat{x_{j}},\cdots,x_{m}),\cdots,\widehat{x_{j}},\cdots,x_{m}%
a_{m}^{j}(x_{1},\cdots,\widehat{x_{j}},\cdots,x_{m})),
\]
where $a_{i}^{j}$ is a holomorphic function in a neighborhood of the origin
$0\in\mathbb{C}^{m-1}$, for all $i=1,\ldots,\widehat{j},\ldots,m$ (the hat
$\widehat{x_{j}}$ stands for omitting that coordinate). The map $h_{j}$ has
linear part given by $Dh_{j}(0)\cdot(x_{1},\cdots,\widehat{x_{j}},\cdots
,x_{n})=(\exp(\frac{\lambda_{1}}{\lambda_{j}})\cdot x_{1},\cdots
,\widehat{x_{j}},\cdots,\exp(\frac{\lambda_{n}}{\lambda_{j}})\cdot x_{n})$. If
the map $h_{j}$ has set of periodic orbits with positive measure, then it is
periodic as we have already seen above (cf. Lemmas
\ref{bounded positive measure} and \ref{Lemma:finitelocalholonomy}). This
implies that the quotients $\frac{\lambda_{1}}{\lambda_{j}}$ are rational
numbers for every $i\neq j$. Therefore, up to dividing the generating vector
field $X$ by a suitable complex number, we may assume that each eigenvalue
$\lambda_{j}$ is a rational number. Such a Siegel singularity will be called
\textit{resonant}. Resonant Siegel type singularities exhibit at most one
non-dicritical separatrix. In dimension $m=3$, a resonant Siegel type
singularity exhibits exactly one \textit{non-dicritical} separatrix, i.e., one
corresponding to an eigenvalue of signal different from the signal of other
two eigenvalues. That is the key point in the discussion that follows.

\subsection{Analytic linearization of Siegel
singularities\label{section:elizarov}}

We end this paragraph dealing with the analytic linearization of Siegel
singularities. Consider a germ of a Siegel singularity at the origin
$0\in\mathbb{C}^{m}$. In the $2$-dimensional case (i.e., $m=2$), the local
holonomy of a separatrix gives the full analytic classification of the
singularity (cf. e.g. \cite{IlyYak2006}). Nevertheless, this is more delicate
in case $m\geq3$. Indeed, a Siegel type singularity may look like a dicritical
singularity when restricted to suitable invariant planes.

In order to better describe generic isolated singularities in the
$m$-dimensional case ($m\geq3$) we need some notation.

\begin{definition}
[Condition $(\star)$]\label{definition:conditionstar} Let $X$ be a germ of a
holomorphic vector field at the origin such that the origin $0\in
\mathbb{C}^{m}$ is a singularity in the Siegel domain. We say that $X$
satisfies condition $(\star)$ if there is a real line $L\subset\mathbb{C}$
through the origin avoiding all the eigenvalues of $X$ such that one of the
connected components of $\mathbb{C}\setminus L$ contains just one eigenvalue
of $X$.
\end{definition}

The above condition holds for $X$ if and only if holds for any vector field
$Y$ such that $X$ and $Y$ are tangent. Condition $(\star)$ implies that $X$ is
in the Siegel domain, but it is stronger than this last. Denote by
$\lambda(X)$ the isolated eigenvalue of $X$ and by $S_{X}$ its corresponding
invariant manifold (the existence is granted by the classical invariant
manifold theorem). We call $S_{X}$ the \textit{distinguished axis} or
\textit{distinguished separatrix} of $X$. The singularity will be called
\textit{holonomy-linearizable} if the holonomy map associated to the
distinguished separatrix is analytically conjugate to its linear part. The
notions of analytically linearizable and holonomy-linearizable are strongly
related as we will see in what follows.

In \cite{Elizarov} it is proved the following result:

\begin{theorem}
[\cite{Elizarov}]\label{Remark:elizarov}Let $X$ and $Y$ be two normal Siegel
type germs of holomorphic vector fields with an isolated singularity at the
origin $0\in\mathbb{C}^{n}$ satisfying condition $(\star)$. Let $h_{X}$ and
$h_{Y}$ be the holonomies of $X$ and $Y$ relatively to $S_{X}$ and $S_{Y}$,
respectively. Then $X$ and $Y$ are analytically equivalent if and only if the
holonomies $h_{X}$ and $h_{Y}$ are analytically conjugate.
\end{theorem}

For a fibered version of the above result we refer to \cite{Helena}, p. 1656.
Notice that a resonant (i.e., a rational eigenvalues) germ of a holomorphic
vector field in the Siegel domain necessarily satisfies condition $(\star)$,
otherwise it would have all eigenvalues with the same signal (positive or
negative) and would be in the Poincaré domain. Therefore, an immediate
consequence of the above results is the following:

\begin{lemma}
\label{Lemma:linearizingSiegel} Let $X$ be a resonant normal Siegel type
holomorphic vector field germ at the origin $0\in\mathbb{C}^{3}$ satisfying
condition $(\star)$. Then the following conditions are equivalent:

\begin{itemize}
\item[\textrm{(i)}] The germ of foliation $\mathcal{F}(X)$ induced by $X$ is
analytically linearizable.

\item[\textrm{(ii)}] $\mathcal{F}(X)$ is holonomy-linearizable, i.e., the
holonomy map of $\mathcal{F}(X)$ relatively to the separatrix $S_{X}$ tangent
to the eigenspace associated to the distinguished axis of $X$ is analytically
conjugate to its linear part.
\end{itemize}
\end{lemma}


In terms of our notion of measure concentration point we have:

\begin{lemma}
\label{Lemma:poincaresiegel} Let $\mathcal{F}$ be a one-dimensional
holomorphic foliation with non-degenerate singularities in a three-dimensional
complex manifold $M^{3}$. Denote by $\Omega^{0}(\mathcal{F})$ the set of all
leaves which are closed off the singular set $\operatorname{Sing}%
(\mathcal{F})\subset M^{3}$. Given a singularity $p\in M^{3}$, suppose that
$p\in\mathcal{C}_{\mu}(\Omega^{0}(\mathcal{F}))$. Then we have two possibilities:

\begin{itemize}
\item[\textrm{(i)}] The singularity is in the Poincaré domain and it is
analytically linearizable, indeed it is of radial type;

\item[\textrm{(ii)}] The singularity is in the Siegel domain and resonant. If
$p$ is a normal singularity, then $\mathcal{F}$ is analytically linearizable
with rational eigenvalues at this singular point.
\end{itemize}
\end{lemma}

\begin{proof}
From what we have seen in Lemma~\ref{Lemma:poincareholonomylinearizable}, for
any singularity in the Poincaré domain, the leaves not contained in
seperatrices are not closed off the singular set. Therefore, since
$p\in\mathcal{C}_{\mu}(\Omega^{0}(\mathcal{F}))$, the singularity is $\mu
$-dicritical and the result follows.

Now assume the singularity is in the Siegel domain. We may then assume that
the coordinate hyperplanes are invariant as well as the coordinate axes, which
are supposed to be tangent to the eigenvectors of (the linear part at the
origin of) the generating vector field. Since $p\in\mathcal{C}_{\mu}%
(\Omega^{0}(\mathcal{F}))$, for any arbitrarily small neighborhood $U$ of $p$
in $M$, the set of leaves of $\left.  \mathcal{F}\right\vert _{U}$ closed off
$p$ has positive measure. Therefore, this set is not contained in the
coordinate hyperplanes or any countable set of hypersurfaces. Thus, the
holonomy maps $h_{j},$ $j=1,\ldots,m$, associated to the coordinate axes, have
positive measure sets of periodic orbits. From Lemma
\ref{bounded positive measure}, this implies that the maps $h_{j}$ are
periodic (and therefore, analytically linearizable) and the eigenvalues of the
generating vector field can be assumed to be rational numbers. Hence, the
singularity is resonant and holonomy-analytically linearizable.
\end{proof}

\section{Dulac correspondence at a Siegel type corner}

\label{section:Dulac}

The framework in this section is motivated by Lemma~\ref{Lemma:poincaresiegel}%
. Let $\mathcal{F}$ be a one-dimensional foliation defined in a neighborhood
of the origin $0\in\mathbb{C}^{m}$, which is assumed to be a normal Siegel
type singularity. We can choose small polydiscs $U\ni0$ centered at the origin
endowed with adapted local coordinates $x=(x_{1},\cdots,x_{m})\in U$ as in
~§\ref{subsubsection:siegel}. Again we fix disc type transverse sections
$\Sigma_{j}=\{x_{j}=a_{j}\}\subset U$ for some sufficiently close to the
origin $a_{j}\in\mathbb{C}$, $\,j=1,\ldots,n$ and denote the intersection
points by $q_{j}=\Sigma_{j}\cap\mathcal{O}_{x_{j}}$.

Fix a holomorphic (generator) vector field $X$ in $U$ with an isolated
singularity at $p$ defining the restriction $\mathcal{F}(U)$. Denote by
$\lambda_{j}\in\mathbb{C}$ the eigenvalue of $DX(0)$ corresponding to the
eigenvector tangent to the $\mathcal{O}_{x_{j}}$ axis and write $X$ as in
equation~(\ref{equation:Siegelvectorfield}). From now on, we assume that $0$
is an \textit{holonomy-analytically linearizable singularity with rational
eigenvalues}.

The Dulac correspondence will be defined as a correspondence $\mathcal{D}%
_{ij}$ from certain subsets of $\Sigma_{i}$ onto certain subsets of
$\Sigma_{j}$ as follows (cf. \cite{Camacho-Scardua, scarduaJDCS}).

The general idea is motivated by the following two dimensional picture: Since
$\lambda_{1}\cdot\lambda_{2}<0$, the topological analytic description of
Siegel plane singularities says that any leaf $L_{z}\in\mathcal{F}$ passing
through $z\in\Sigma_{i}$ must intersect $\Sigma_{j}$ provided that $z$ is
close enough to the origin of $\Sigma_{i}$. Therefore, we associate to the
intersection points $L_{z}\cap\Sigma_{i}$, the intersection points $L_{z}%
\cap\Sigma_{j}$. We shall write $\mathcal{D}_{ij}(z)$ to denote this subset
$L_{z}\cap\Sigma_{j}$ just for simplicity. The Dulac correspondence is a
multivalued correspondence ${\mathcal{D}}_{ij}:\Sigma_{i}\rightarrow\Sigma
_{j}$, which is obtained by following the local leaves of $\left.
\mathcal{F}\right\vert _{U}$ from $\Sigma_{i}$ to $\Sigma_{j}$.

Next we describe the Dulac correspondence on each specific case we need. The
main hypothesis being that the singularity is Siegel resonant, normal and holonomy-linearizable.

\subsection{Dimension two}

Suppose $m=2$. In this case holonomy-analytically linearizable foliations and
analytically linearizable foliations are equivalent notions. Thus, we are
actually assuming that the origin is a linearizable singularity in the Siegel
domain. We may then choose local holomorphic coordinates $(x,y)\in U$ such
that the local separatrices $D_{i}$ and $D_{j}$ through the singularity are
given by $D_{i}:(x=0)$, \thinspace\ $D_{j}:(y=0)$, \thinspace\ and such that
$\left.  \mathcal{F}\right\vert _{U}$ is given by $\lambda xdy-\mu ydx=0$,
$q_{o}:x=y=0$, where $\lambda,\mu\in\mathbb{Q}$ and $\frac{\lambda}{\mu}%
\in\mathbb{Q}_{-}$. We fix the local transverse sections as $\Sigma_{j}=(x=1)$
and $\Sigma_{i}=(y=1)$, such that $\Sigma_{i}\cap D_{i}=q_{i}\neq q_{o}$ and
$\Sigma_{j}\cap D_{j}=q_{j}\neq q_{o}$. Let us denote by $h_{o}\in
\operatorname*{Diff}(\Sigma_{i},q_{i})$ the local holonomy map of the
separatrix $D_{i}$ corresponding to the corner $q_{o}$. Then we have
$h_{o}(x)=\exp(2\pi\frac{\lambda}{\mu}\sqrt{-1})\cdot x$. The Dulac
correspondence is therefore given by
\[
{\mathcal{D}}_{ij}\colon(\Sigma_{i},q_{i})\rightarrow(\Sigma_{j}%
,q_{j}),{\mathcal{D}}_{ij}(x_{o})=x_{o}^{\frac{\mu}{\lambda}}.
\]

\subsection{Dimension three}

We may assume that the eigenvalues of the linear part of $X$ are $\lambda
_{1}\in\mathbb{Q}_{-}$ and $\lambda_{2},\lambda_{3}\in\mathbb{Q}_{+}$. The
coordinate plane $H_{1}=E(x_{2}x_{3})=\{x_{1}=0\}$ is invariant by
$\mathcal{F}(U)$ and we denote this foliation by $\mathcal{F}(U)_{x_{2}x_{3}}%
$. Furthermore, this planar foliation is analytically linearizable and
\textit{absolutely dicritical}, i.e., all leaves are contained in separatrices
and accumulate at the origin. In particular, the closure of each leaf
$L\in\mathcal{F}(U)_{x_{2}x_{3}}$ is a (unique) separatrix $\Gamma
=L\cup\{0\}\subset E(x_{2}x_{3})$ through the origin. By the
analytic-linearization (or, more generally, by the topological analytic
description of Siegel singularities), we know that there is a germ at the
origin of an analytic surface $H(x_{1},\Gamma)$ that is invariant by the
vector field $X$, contains the axis $\mathcal{O}_{x_{1}}$, and the curve
$\Gamma$. The surface $H(x_{1},\Gamma)$ meets the disc $\Sigma_{1}$
transversely at a $1$-disc $\Sigma_{1}(\Gamma)$ centered at $q_{1}$. Also,
given any point $q_{2}\in\Gamma\setminus\{0\}$ and a transverse $2$-disc
$\Sigma_{2}$ centered at $q_{2}$, the surface $H(x_{1},\Gamma)$ meets
$\Sigma_{2}$ transversely at a $1$-disc $\Sigma_{2}(\Gamma)$. Reasoning as in
the dimension two case above, we obtain a Dulac correspondence $\mathcal{D}%
_{1,\Gamma}\colon\Sigma_{1}(\Gamma)\rightarrow\Sigma_{2}(\Gamma)$ by following
the local leaves of the foliation on the invariant variety $H(x_{1},\Gamma)$.
\vglue.1in A pair of separatrices of a non-degenerate Siegel type singularity
is called \textit{Siegel pair} if the quotient of the eigenvalues
corresponding to these separatrices is a negative real number. As a
consequence of the above considerations, we obtain the following complement to
Lemma~\ref{Lemma:poincaresiegel}:

\begin{lemma}
\label{Lemma:dulacaccumulation} Let $\mathcal{F}$ be a one-dimensional
holomorphic foliation with non-degenerate singularities in a complex manifold
$M$. Let $\mathcal{X}(\mathcal{F})\subset M\setminus\operatorname{Sing}%
(\mathcal{F})$ be any invariant set of leaves. Let $p\in\operatorname{Sing}%
(\mathcal{F})\subset M$ be a normal Siegel type singularity. Given a Siegel
type pair of separatrices $S_{p},S_{p}^{\prime}$ we have $S_{p}\subset
\mathcal{C}_{\mu}(\mathcal{X}(\mathcal{F}))$ if and only if $S_{p}^{\prime
}\subset\mathcal{C}_{\mu}(\mathcal{X}(\mathcal{F}))$, where $\mathcal{C}_{\mu
}(\mathcal{X}(\mathcal{F}))$ denotes the set of measure concentration points
of the set $\mathcal{X}(\mathcal{F})$.
\end{lemma}

\section{Stable graphs}

In order to motivate our notion below, we recall the classical real framework.
Let $X$ be a smooth vector field in a real surface $N^{2}$ with isolated
singularities. By a \textit{graph} for $X$, we mean an invariant compact
subset $\Gamma\subset N$ consisting of singularities and orbits such that the
$\alpha$-limit and the $\omega$-limit of each orbit of $X$ contain some
singularity in the graph. This notion admits a natural extension to the
complex two-dimensional case as follows:

Let $\mathcal{F}$ be a one-dimensional foliation with non-degenerate
singularities in a complex surface $M^{2}$.

\begin{definition}
[graph in complex surfaces]\label{Definition:stablegraph} A graph of
$\mathcal{F}$ is an invariant compact connected analytic subset $\Gamma\subset
M$ of pure-dimension one such that:

\begin{enumerate}
\item The singularities of $\mathcal{F}$ in $\Gamma$ are all of Siegel type;

\item Each leaf $L\subset\Gamma$ is contained in an analytic curve and
accumulates at some singularity $p\in\operatorname*{Sing}(\mathcal{F}%
)\cap\Gamma$;

\item Any local separatrix through $p\in\Gamma\cap\operatorname*{Sing}%
(\mathcal{F})$ is contained in $\Gamma$.
\end{enumerate}
\end{definition}

Next we obtain the following stability theorem for graphs:

\begin{proposition}
\label{Proposition:stablegraphstability} Let $\mathcal{F}$ be a holomorphic
foliation of dimension one with non-degenerate singularities in a compact
surface $M^{2}$. Suppose $\mathcal{F}$ has a stable graph $\Gamma\subset M$,
then there is a fundamental system of invariant neighborhoods $W$ of $\Gamma$
in $M$ such that each leaf intersecting $W$ is quasi-compact. In such a
neighborhood the foliation admits a holomorphic first integral.
\end{proposition}

\begin{proof}
The proof is somehow similar to the proof of the main result in
\cite{Mattei-Moussu}, where it is proved the existence of a holomorphic first
integral for a germ of a non-dicritical foliation in dimension two, provided
that the leaves are closed off the singular point. In \cite{Mattei-Moussu} it
is used an induction argument, since the singularity is not necessarily
non-degenerate. This is not the case here. Nevertheless, their construction of
invariant neighborhoods and holomorphic first integrals, from the finiteness
of the combined holonomy groups of the projective lines in the exceptional
divisor of the resolution of singularities, can be repeated here with minor
changes. Let us give the main steps. First we remark that all singularities in
$\Gamma$ are linearizable of the local form $nxdy-mydx=0,\,n,m\in
\mathbb{Z}_{+}$. Therefore, each such a singularity exhibits a local
holomorphic first integral of the form $f=x^{m}y^{n}$. Now, the finiteness of
the holonomy groups allow the \textit{holonomy extension} of these first
integrals to a neighborhood of each irreducible component $\Gamma_{j}%
\subset\Gamma$ of the graph $\Gamma$. Here the procedure is the same as in the
case of a single blow-up in \cite{Mattei-Moussu}. This gives a holomorphic
first integral $f_{j}$ defined in a neighborhood $W_{j}$ of $\Gamma_{j}$ in
$M$, in such a way that the union $W=\bigcup\nolimits_{j}W_{j}$ is an
invariant neighborhood of $\Gamma$. Moreover, thanks to the finiteness of the
virtual holonomy groups, i.e., of the combined holonomy groups of the
components of $\Gamma$, the first integrals $f_{j}$ are such that each corner
$q\in\Gamma_{i}\cap\Gamma_{j}\neq\emptyset$ is a singularity and in some
neighborhood of this singularity we have $f_{i}^{n_{ij}}=f_{j}^{m_{ij}}$ for
$n_{ij},m_{ij}\in\mathbb{N}$. This shows the existence of a holomorphic first
integral $f$ in $W$, which is defined on each $W_{j}$ by an expression like
$\left.  f\right\vert _{W_{j}}=f_{j}^{\nu_{j}}$ for a suitable $\nu_{j}%
\in\mathbb{N}$.
\end{proof}

\begin{proof}
[Proof of Theorem~B]Let $\mathcal{F}$ be as in Theorem~B and denote by
$\Omega(\mathcal{F})\subset M$ the union of quasi-compact leaves of
$\mathcal{F}$. Since $\Omega(\mathcal{F})$ has positive measure, then arguing
as in Lemma~\ref{Lemma:concentrationinside}, we can assure the existence of a
leaf $L_{0}\subset\mathcal{C}_{\mu}(\Omega(\mathcal{F}))$. On the other hand,
we conclude from Lemma~\ref{bounded positive measure} the finiteness of the
holonomy group of $L_{0}$. Analogously, the virtual holonomy group of $L_{0}$
is finite. Finally, applying induction and the existence of first integrals
nearby each separatrix meeting $L_{0}$ (\cite{Mattei-Moussu}), we construct a
stable graph $\Gamma$ for $\mathcal{F}$.
\end{proof}

\begin{example}
\textrm{Let $X$ be the polynomial vector field $X=x\frac{\partial}{\partial
x}+y\frac{\partial}{\partial y}+\lambda z\frac{\partial}{\partial z}$ on
$\mathbb{C}^{3}$, where $\lambda\in\mathbb{R}\setminus\mathbb{Q}$. Then the
hyperplane $H\subset\mathbb{C}\mathbb{P}^{3}$ given by $\Gamma\cap
\mathbb{C}^{3}=\{z=0\}$ is invariant by the foliation $\mathcal{F}$ induced by
$X$ in $\mathbb{C}\mathbb{P}^{3}$. Moreover: (1) the restriction }$\left.
\mathrm{\mathcal{F}}\right\vert $\textrm{$_{H}$ is biholomorphically
equivalent to the \textquotedblleft radial foliation\textquotedblright%
\ induced in $\mathbb{C}\mathbb{P}^{2}$ by the vector field $\vec
{R}(x,y)=x\frac{\partial}{\partial x}+y\frac{\partial}{\partial y}$. (2) On
the other hand, since $\lambda$ is irrational, the holonomy groups of leaves
in $H$ are not finite. Thus, though algebraic, the leaves in $H$ do not fit
into the above framework. }
\end{example}

Now we extend the above notions to dimension $m\geq2$. For this sake, we shall
introduce some notation: Let $\mathcal{F}$ be a one-dimensional foliation with
a finite number of singularities in a complex manifold $M^{m}$ of dimension
$m\geq2$. Given a leaf $L$ of $\mathcal{F}$, we say that a leaf $L_{1}$ is in
\textit{Siegel pairing} with $L$ if there is a singularity $p\in\overline
{L}\cap\overline{L_{1}}$ such that $L$ and $L_{1}$ induce distinct
separatrices $S(L)_{p}$ and $S(L_{1})_{p}$ at $p$, which are of Siegel pair
type. Then we denote by $\Gamma_{1}(L)$ the union of $L$ and all such leaves
$L_{1}$ which are in Siegel pairing with $L$. In the same way, we consider all
leaves which are in Siegel pairing with the leaves in $\Gamma_{1}(L)$ and its
union will be denoted by $\Gamma_{2}(L)$. Since the singular set of
$\mathcal{F}$ is finite, we may proceed this way and by induction we obtain a
subset $\Gamma(L)\subset M$, which shall be called the \textit{Siegel
component} of $L$.

\begin{definition}
[graph in complex manifolds]\label{Definition:stablegraphanydimension} Suppose
all the singularities of $\mathcal{F}$ are non-degenerate. A graph of
$\mathcal{F}$ is an invariant compact connected analytic subset $\Gamma\subset
M$ of pure-dimension one such that $\Gamma=\Gamma(L)$, the Siegel component of
some leaf $L$ of $\mathcal{F}$. The graph $\Gamma\subset M$ of $\mathcal{F}$
is called stable if each virtual holonomy group of $\Gamma$ is finite in the
sense of Definition~\ref{Definition:finitevirtualholonomy}.
\end{definition}

In particular, all the Siegel type singularities of $\mathcal{F}$ in $\Gamma$
are resonant and holonomy-linearizable.

\section{One-dimensional foliations in complex projective spaces}

Since the Cousin multiplicative problem always admits a solution in
$\mathbb{C}^{m+1}\setminus\{0\},\,m\geq2$, a one-dimensional holomorphic
foliation with singularities $\mathcal{F}$ in $\mathbb{C}\mathbb{P}^{m}$ is
always defined in any affine space $\mathbb{C}^{m}\subset\mathbb{C}%
\mathbb{P}^{m}$ by a polynomial vector field with isolated singularities. From
now on, by \emph{foliation} we shall mean a one-dimensional holomorphic
foliation with singularities. In this paper, by an \textit{algebraic leaf} we
mean a leaf $L\subset\mathbb{C}\mathbb{P}^{m}$ of the (non-singular)
foliation, such that the closure $\overline{L}\subset\mathbb{C}\mathbb{P}^{m}$
is algebraic of dimension one. Equivalently, $L$ is contained in an algebraic
curve. In this case, $\Lambda(L):=\overline{L}$ is an algebraic invariant
curve and we have $\Lambda(L)\setminus L=\overline{L}\cap\operatorname{Sing}%
(\mathcal{F})\subset\operatorname{Sing}(\mathcal{F})$. Given a foliation
$\mathcal{F}$ in $\mathbb{C}\mathbb{P}^{m}$, we denote by $\Omega
(\mathcal{F})$ the \emph{collection} of all algebraic leaves of $\mathcal{F}$.
This is therefore a collection of leaves of $\mathcal{F}$, not a subset of
$\mathbb{C}\mathbb{P}^{m}$. Given a leaf $L\subset\operatorname*{Sep}%
(\mathcal{F})$, we denote by $\Lambda=\Lambda(L)$ the algebraic curve
$\overline{L}\subset L\cup\operatorname{Sing}(\mathcal{F})\subset
\mathbb{C}\mathbb{P}^{m}$, i.e., the corresponding algebraic curve containing
the leaf $L$. Then we shall denote by $\Omega(\mathcal{F})\subset
\mathbb{C}\mathbb{P}^{m}$ the union of all such algebraic curves $\Lambda(L)$
with $L\in\operatorname*{Sep}(\mathcal{F})$.

Recall that a Leaf $L$ of $\mathcal{F}$ is \textit{quasi-compact} if it is
closed off the singular set of the $\mathcal{F}$, \textit{i.e.,} $\overline
{L}\setminus L\subset\operatorname{Sing}(\mathcal{F})$. The following simple
criterion will be useful for our intents:

\begin{lemma}
\label{Lemma:Remmert-steinleaf} A leaf $L$ of a one-dimensional holomorphic
foliation $\mathcal{F}$ in $\mathbb{C}\mathbb{P}^{m}$ is algebraic if and only
if it is quasi-compact.
\end{lemma}

\begin{proof}
According to a theorem of Chow (\cite{Gunning-Rossi}), complex analytic
subsets of projective spaces are indeed algebraic subsets. Therefore, a leaf
$L\subset\mathbb{C}\mathbb{P}^{m}$ is algebraic if and only if its closure
$\overline{L}\subset\mathbb{C}\mathbb{P}^{m}$ is an analytic subset of
dimension one. In addition, by the classical extension theorem of Remmert and
Stein (\cite{Gunning-Rossi}), the closure $\overline{L}\subset\mathbb{C}%
\mathbb{P}^{m}$ is analytic of dimension one if and only if $\overline
{L}\setminus L$ is contained in a analytic subset of dimension zero.
Therefore, a leaf $L$ of $\mathcal{F}$ is algebraic if and only if
$\overline{L}\setminus L\subset\operatorname{Sing}(\mathcal{F})$.
\end{proof}

\subsection{Algebraic leaves and finite order}

\begin{definition}
[\cite{scarduaIndagationes}]\textrm{Let $\mathcal{F}$ be a foliation of
codimension $k$ in a manifold $M$ (perhaps non-compact). A \textit{compact
total transverse section} of $\mathcal{F}$ is a compact $k$-manifold $T\subset
M$ (possibly with boundary) such that every leaf of $\mathcal{F}$ intersects
the interior of $T$.}
\end{definition}

\begin{lemma}
[\cite{scarduaIndagationes}]\label{Lemma:section} Let $\mathcal{F}$ be a
holomorphic foliation of dimension $1$ in $\mathbb{C}\mathbb{P}^{m}$ with
(finite) singular set $\operatorname{Sing}(\mathcal{F})\subset\mathbb{C}%
\mathbb{P}^{m}$. There exists a finite collection of immersed closed discs
$\overline{D_{j}}\subset\mathbb{C}\mathbb{P}^{m},\,j=1,...,r$ pairwise
disjoint such that\textrm{:} \textrm{(i)} $\overline{D_{j}}\approx
\{z\in{\mathbb{C}}^{n-1}:\left\vert z\right\vert \leq1\}$ and $D_{j}%
=\overline{D_{j}}\setminus\partial\overline{D_{j}}$ is transverse to
$\mathcal{F}$. \textrm{(ii)} Each leaf of $\mathcal{F}$ intersects at least
one of the open discs $D_{j}$. In other words, the manifold with boundary
$T=\overline{D_{1}}\cup\cdots\cup\overline{D_{r}}$ is a compact total
transverse section to $\mathcal{F}$.
\end{lemma}

\begin{proof}
Denote by $\{p_{1},\cdots,p_{r}\}$ the singular set of $\mathcal{F}$. Choose
small neighborhoods $U_{j}\ni p_{j}$ diffeomorphic to polydiscs $\{z\in
{\mathbb{C}}^{n}:\left\vert z\right\vert \leq1\}$ centered at $p_{j}$. Then
$X:=\mathbb{C}\mathbb{P}^{m}\setminus\cup_{j=1}^{r}U_{j}$ is compact.
Therefore, there exists an open cover $\cup_{\alpha\in A}U_{\alpha}%
=\mathbb{C}\mathbb{P}^{n}\setminus\operatorname{Sing}(\mathcal{F})$ by
distinguished neighborhoods $U_{\alpha}$, such that in each $U_{\alpha}$ we
have an embedded disc $\Sigma_{\alpha}\approx\mathbb{D}$ transverse to
$\left.  \mathcal{F}\right\vert _{U_{\alpha}}$ satisfying the following
property: if $p\in U_{\alpha}$, then $L_{p}\cap\Sigma_{\alpha}\neq\emptyset$.
Since $X$ is compact, there exists a finite subcovering $X\subset
U_{\alpha_{1}}\cup\cdots\cup U_{\alpha_{\ell}}$. In particular, for any $p\in
X$, the leaf $L_{p}$ intersects some $\Sigma_{\alpha_{j}}$, $j\in
\{1,\ldots,\ell\}$. On the other hand, a leaf cannot remain in a polydisc
$U_{j}$ for it cannot be bounded in the affine space. Therefore, any leaf
intersects $X$ and thus the interior of some $\Sigma_{j}$.
\end{proof}

Therefore, Lemmas ~\ref{Lemma:transvuniform} and \ref{Lemma:section} motivate
the following:

\begin{definition}
[relative order of a leaf]\label{Definition:fininteorder} \textrm{Let
$\mathcal{F}$ be a one-dimensional foliation in $\mathbb{C}\mathbb{P}^{m}$
with finite singular set $\operatorname{Sing}(\mathcal{F})\subset
\mathbb{C}\mathbb{P}^{m}$. We say that a leaf $L\in\mathcal{F}$ has
\textit{finite order} if for some compact total transverse section
$T\subset\mathbb{C}\mathbb{P}^{m}$ to $\mathcal{F}$ the intersection $L\cap T$
is a finite set. Fixed such section $T$, the number of intersection points
between }$\mathrm{L}$\textrm{ and }$\mathrm{T}$ \textrm{is called the
\textit{relative order of }}$\mathrm{\mathit{L}}$\textrm{\textit{ with respect
to }}$\mathrm{\mathit{T}}$\textrm{\textit{ and denoted by }$\operatorname{ord}%
(L,T)$, i.e., $\operatorname{ord}(L,T)=\#(L\cap T)$.}
\end{definition}

\begin{remark}
\textrm{Even though Lemma~\ref{Lemma:transvuniform} gives some intrinsic
character to the notion of order, it is not clear \textit{a priori} that this
is really intrinsic. For this reason, we shall \textit{not} consider the
notion of \textit{order} of a leaf $L$ of $\mathcal{F}$. Nevertheless, for our
purposes it is enough to observe the following: from
Lemma~\ref{Lemma:transvuniform}, given a leaf $L$ of $\mathcal{F}$, a point
$q\in L$ and a transverse disc $\Sigma$ to $\mathcal{F}$ centered at a point
$q\in L$ we have (see the proof of Claim~\ref{Claim:finiteordercompact})
\[
\#(L\cap\Sigma)\leq\operatorname*{ord}(L,\mathcal{F}).
\]
}
\end{remark}

\begin{lemma}
\label{Lemma:finiteorderisalgebraic} Given a holomorphic foliation
$\mathcal{F}$ of dimension $k$ in $\mathbb{C}\mathbb{P}^{m}$ and a compact
total transverse section $T\subset\mathbb{C}\mathbb{P}^{m}$ to $\mathcal{F}$,
a leaf $L\in\mathcal{F}$ is algebraic if and only if it has finite
\textrm{(}relative\textrm{)} order (i.e., $\operatorname{ord}(L,T)<\infty$).
More precisely, we have $\Omega(\mathcal{F})=\bigcup\limits_{n\in\mathbb{N}%
}\Omega(\mathcal{F},T,n)$, where%
\[
\Omega(\mathcal{F},T,n):=\{L\in\mathcal{F}:\#(L\cap T)\leq n\}.
\]

\end{lemma}

\begin{proof}
According to Lemma~\ref{Lemma:Remmert-steinleaf}, a leaf $L\subset
\mathbb{C}\mathbb{P}^{m}$ is algebraic if and only if $\overline{L}\setminus
L\subset\operatorname{Sing}(\mathcal{F})$. If $L$ is algebraic, then
$\overline{L}$ is a compact algebraic curve in $\mathbb{C}\mathbb{P}^{m}$.
Since $\overline{L}\cap T=L\cap T$ and $\overline{L}$ is compact, then
$\#(L\cap T)<\infty$. Therefore, every algebraic leaf has finite relative
order. Conversely, let $L\in\mathcal{F}$ be a leaf with finite relative order
with respect to $T$. In particular, by the same arguments used in the proof of
Claim~\ref{Claim:finiteordercompact}, the leaf $L$ is closed off the singular
set. Therefore, by the initial remark, $L$ is an algebraic leaf.
\end{proof}

Another simple but very useful remark is the following:

\begin{lemma}
\label{Lemma:boundedorderalgebraic} Let a leaf $L\in\mathcal{F}$ be such that
$L\subset{\mathcal{C}}_{\mu}(\Omega(\mathcal{F},T,k))$ for some $k\in
\mathbb{N}$ and some total transverse section $T\subset\mathbb{C}%
\mathbb{P}^{m}$. Then $L$ is algebraic.
\end{lemma}

\begin{proof}
Suppose by contradiction that $L\subset{\mathcal{C}}_{\mu}(\Omega
(\mathcal{F},T,k))$ is not algebraic. Then Lemma~\ref{Lemma:Remmert-steinleaf}
says that $L$ is not closed off $\operatorname{Sing}(\mathcal{F})$. Therefore,
there is a non-singular accumulation point $q_{\infty}\in\overline{L}\setminus
L$, $q_{\infty}\notin\operatorname{Sing}(\mathcal{F})$. Given an
\textit{arbitrarily small} transverse disc $\Sigma_{q_{\infty}}$ to
$\mathcal{F}$ centered at $q_{\infty}$, we have $\#(L\cap\Sigma_{q_{\infty}%
})=\infty$. Then there is a disc $T_{j}\subset T$ such that $L_{q_{\infty}}$
meets $T_{j}$ at an interior point, say $q\in T_{j}$. Choose now a point $p\in
L$ and a transverse disc $\Sigma_{p}$ centered at $p$. By the Transverse
uniformity lemma (Lemma~\ref{Lemma:transvuniform}), there is a map from the
disc $\Sigma_{p}$ to $\Sigma_{q_{\infty}}$ and thus to the disc $T_{j}$. We
conclude that for any $w\in\Sigma_{p}$ close enough to $p$, we have
$\#(L_{w}\cap T_{j})\geq k!$. In particular, $L_{w}$ has order greater than
$k$. This shows that $\operatorname*{med}(\Sigma_{p}\cap\Omega(\mathcal{F}%
,T,k))=0$, yielding a contradiction to $L\subset\mathcal{C}_{\mu}%
(\Omega(\mathcal{F},T,k))$.
\end{proof}

We propose the following conjecture:

\begin{conjecture}
\label{Conjecture:stablegraphproperties} Let $\mathcal{F}$ be a holomorphic
foliation of dimension one with non-degenerate singularities in the complex
projective space $\mathbb{C}\mathbb{P}^{3}$ admitting a stable graph
$\Gamma(L)$, where $L\in\mathcal{F}$ is a stable algebraic leaf of
$\mathcal{F}$. Then all leaves are algebraic.
\end{conjecture}




\section{Proof of the Algebraic Stability theorem}

In what follows, we consider the following situation: $\mathcal{F}$ is a
holomorphic foliation of dimension one with non-degenerate singularities in
the complex projective space $\mathbb{C}\mathbb{P}^{m}$ such that the set
$\Omega(\mathcal{F})$, union of all algebraic leaves of $\mathcal{F}$, has
positive measure. From Lemmas~\ref{Lemma:nondiscrete} and
~\ref{Lemma:invariant} the set $\mathcal{C}_{\mu}(\Omega(\mathcal{F}))$ is
invariant and contains some leaf, say $L_{0}\subset\mathcal{C}_{\mu}%
(\Omega(\mathcal{F}))$. However, this is not enough for our purposes, since we
cannot control the degree (or the order) of the algebraic leaves accumulating
at $L_{0}$. Actually, we may say more: from Lemma~\ref{Lemma:section}, we may
choose a compact total transverse section $T\subset\mathbb{C}\mathbb{P}^{m}$
to $\mathcal{F}$. Using then Lemmas \ref{Lemma:concentrationgraduation} and
\ref{Lemma:finiteorderisalgebraic} we immediately conclude that:

\begin{lemma}
\label{Lemma:concentrationleaf} Under the hypothesis of Theorem~C, there is a
leaf $L_{0}\in\mathcal{F}$ such that $L_{0}\subset{\mathcal{C}}_{\mu}%
(\Omega(\mathcal{F},T,k))\neq\emptyset$ for some $k\in\mathbb{N}$ and some
total transverse section $T\subset\mathbb{C}\mathbb{P}^{m}$.
\end{lemma}

From now on, we assume the dimension is $m=3$. In this case, a Siegel
singularity $p\in\operatorname{Sing}(\mathcal{F})$ exhibits a distinguished
separatrix, also called \textit{non-dicritical} separatrix, and denote by
$S_{p}$. Using then Lemmas~\ref{Lemma:poincaresiegel} and
\ref{Lemma:dulacaccumulation} we obtain:

\begin{lemma}
\label{Lemma:separatrixhasfiniteholonomy} Assume the dimension is $m=3$. Let
$p\in{\mathcal{C}}_{\mu}{(\Omega(\mathcal{F},T,k))}\cap\operatorname{Sing}%
(\mathcal{F})$ be a Siegel type singularity, which is a measure concentration
point of $\Omega(\mathcal{F},T,k)$. Then the singularity of $\mathcal{F}$ at
$p$, say $\mathcal{F}_{p}$, is resonant. Moreover, the non-dicritical (i.e.,
the distinguished) separatrix $S_{p}$ of $\mathcal{F}$ through $p$ is also
contained in the set of measure concentration points of $\Omega(\mathcal{F}%
,T,k)$ and has finite holonomy map with respect to the singularity
$\mathcal{F}_{p}$. In particular, if $\mathcal{F}_{p}$ is normal, then it is a
resonant analytically linearizable singularity admitting holomorphic first integrals.
\end{lemma}

\begin{proof}
Since each algebraic leaf is closed off the singular set, we have
$\Omega(\mathcal{F},T,k)\subset\Omega(\mathcal{F})\subset\Omega^{0}%
(\mathcal{F})$ in the notation of Lemma~\ref{Lemma:poincaresiegel}. The result
follows from Lemmas~\ref{Lemma:poincaresiegel} and
\ref{Lemma:dulacaccumulation} as well as the considerations in §§
\ref{subsubsection:poincare} and ~\ref{subsubsection:siegel}. In fact, since
it is (non-degenerate and) non-$\mu$-dicritical, $p$ is a singularity
necessarily in the Siegel domain. Then, for instance by the topological
description of Siegel type singularities (\cite{CKP}, \cite{Brunellasiegel})
we conclude that any positive measure invariant set $X$ that accumulates at
the singularity $p$, also accumulates at any non-dicritical separatrix of the
foliation through $p$: indeed, there are two possibilities; either $X$ is
contained in some proper invariant analytic variety $\Gamma(X)$, that contains
the singularity, or $X$ accumulates at some separatrix of $\mathcal{F}$
through $p$. In the first case $X$ has zero measure and, \textit{a fortiori},
we have $p\not \in \mathcal{C}_{\mu}(X)$. These considerations prove some
separatrix $S_{p}$ of $\mathcal{F}$ through $p$ is also contained in
$\mathcal{C}_{\mu}(\Omega(\mathcal{F},T,k))$.
Lemma~\ref{Lemma:finitelocalholonomy} then assures the finiteness of the
holonomy maps of the separatrices. This already says that $p$ is a resonant
holonomy-linearizable singularity.



\end{proof}

The next step is the linearization of all (normal) singularities, as follows:

\begin{lemma}
\label{Lemma:analyticallylinearizablesing} Let $p\in{\mathcal{C}}_{\mu
}{(\Omega(\mathcal{F},T,k))}\cap\operatorname*{Sing}(\mathcal{F})$ be a
non-degenerate singularity which is a measure concentration point of
$\Omega(\mathcal{F},T,k)$. Then $p\in\operatorname{Sing}(\mathcal{F})$ is
analytically linearizable of radial type or it is a holonomy-linearizable
singularity of resonant Siegel type, which is analytically linearizable in
case it is normal. Given a leaf $L_{0}\subset{\mathcal{C}}_{\mu}%
(\Omega(\mathcal{F},T,k))$ all the singularities $q\in\operatorname{Sing}%
(\mathcal{F})\cap\Gamma(L_{0})$ are analytically linearizable with rational eigenvalues.
\end{lemma}

\begin{proof}
If the singularity $q\in\operatorname{Sing}(\mathcal{F})\cap\Gamma(L_{0})$ is
in the Poincaré domain, then the analytic linearization follows from
Lemma~\ref{Lemma:poincaresiegel}. On the other hand, if $q$ is a resonant
normal singularity in the Siegel domain, then we apply
Lemma~\ref{Lemma:separatrixhasfiniteholonomy} in order to ensure the
linearization. Finally, the rationality of the eigenvalues comes from
$L_{0}\subset{\mathcal{C}}_{\mu}(\Omega(\mathcal{F},T,k))$.
\end{proof}


\begin{proposition}
\label{Proposition:desingularization} Let $\mathcal{F}$ be as in Theorem~C,
then there exists a finite sequence of quadratic blow-ups $\pi:\widetilde
{M}\rightarrow\mathbb{C}\mathbb{P}^{m}$ at the $\mu$-dicritical singularities
in $\mathcal{C}_{\mu}(\Omega(\mathcal{F}))$ such that $\widetilde{M}$ is a
compact complex manifold and the induced pull-back foliation $\widetilde
{\mathcal{F}}=\pi^{\ast}(\mathcal{F})$ is a foliation with the following properties:

\begin{itemize}
\item[\textrm{(i)}] The singularities of $\widetilde{\mathcal{F}}$ are of
non-degenerate type;

\item[\textrm{(ii)}] The set $\Omega(\widetilde{\mathcal{F}})\subset
\widetilde{M}$, union of algebraic leaves of $\widetilde{\mathcal{F}}$, has
positive measure;

\item[\textrm{(iii)}] The singularities of $\widetilde{\mathcal{F}}$ are
non-$\mu$-dicritical singularities or $\mu$-dicritical singularities which are
not measure concentration points of the set $\Omega(\widetilde{\mathcal{F}%
})\subset\widetilde{M}$ of algebraic leaves of $\widetilde{\mathcal{F}}$.
\end{itemize}
\end{proposition}

\begin{proof}
[Proof of Proposition~\ref{Proposition:desingularization}]Let $\mathcal{F}$ be
as in the statement. By Lemma~\ref{Lemma:poincareholonomylinearizable}, a
singularity $p\in\mathcal{C}_{\mu}(\Omega(\mathcal{F}))$ is either non-$\mu
$-dicritical or of radial type. For a singularity of radial type $p_{1}%
\in\operatorname{Sing}(\mathcal{F})$, a single quadratic (punctual) blow-up at
this singular point, say $\pi_{1}\colon M_{1}\rightarrow\mathbb{C}%
\mathbb{P}^{m}$, produces a pull-back foliation $\mathcal{F}_{1}=\pi_{1}%
^{-1}(\mathcal{F})$ with the following properties:

\begin{enumerate}
\item $\left.  \mathcal{F}_{1}\right\vert _{M_{1}\setminus\pi_{1}^{-1}(p_{1}%
)}$ is equivalent to $\left.  \mathcal{F}\right\vert _{\mathbb{C}%
\mathbb{P}^{m}\setminus\{p_{1}\}}$;

\item The exceptional divisor $\pi_{1}^{-1}(p_{1})\cong\mathbb{C}%
\mathbb{P}^{m-1}$ contains no singularity of $\mathcal{F}_{1}$ and is
everywhere transverse to $\mathcal{F}_{1}$;

\item The set of algebraic leaves of $\mathcal{F}_{1}$ is birationally
equivalent, by the map $\pi\colon M_{1}\rightarrow\mathbb{C}\mathbb{P}^{m}$,
to the set of algebraic leaves of $\mathcal{F}$. In particular, $\Omega
(\mathcal{F}_{1})$ has positive measure if and only if $\Omega(\mathcal{F})$
has positive measure.
\end{enumerate}

By the proper mapping theorem \cite{Gunning-Rossi}, a leaf $L$ of
$\mathcal{F}$ is algebraic if and only if the inverse image $\pi_{1}%
^{-1}(L)\subset M_{1}$ is a finite union of algebraic leaves of $\mathcal{F}%
_{1}$ and (possibly) some singular points. In particular, we have
$\operatorname*{med}(\Omega(\mathcal{F}_{1}))=0$ if and only if
$\operatorname*{med}(\Omega(\mathcal{F}))=0$. Thus, a finite iteration of this
process at the $\mu$-dicritical and measure concentration singular points of
$\mathcal{F}$ gives the desired result.
\end{proof}

\begin{notation}
\textrm{Consider $\mathcal{F}$ and $\widetilde{\mathcal{F}}$ as in the above
proposition. For the sake of simplicity, from now on we shall write
}$\widetilde{\mathbb{CP}^{m}}$\textrm{ as for }$\widetilde{\mathrm{M}}%
$\textrm{ and }$\mathrm{\mathcal{F}}$\textrm{ as for }$\widetilde
{\mathrm{\mathcal{F}}}$\textrm{. This corresponds to say that, after the
blowing-up procedure, \emph{the foliation $\mathcal{F}$ has no $\mu
$-dicritical singularity in} $\mathcal{C}_{\mu}(\Omega(\mathcal{F}))\subset$%
}$\widetilde{\mathbb{CP}^{m}}$\textrm{. \vglue.1in Given an algebraic leaf
$L\in\Omega(\mathcal{F})$ of $\mathcal{F}$, denote by $\Lambda(L)=L\cup
(\overline{L}\cap\operatorname{Sing}(\mathcal{F}))$ the irreducible algebraic
curve containing $L$. In particular, $\Lambda(L)$ is the union of $L$ and all
the local separatrices \textit{tangent} to $L$. If we denote by $\Omega
(\mathcal{F})\subset$}$\widetilde{\mathbb{CP}^{m}}$\textrm{ the closure of
$\Omega(\mathcal{F})$ in }$\widetilde{\mathbb{CP}^{m}}$\textrm{, then%
\[
\Omega(\mathcal{F})=\bigcup\limits_{L\in\Omega(\mathcal{F})}\Lambda
(L)=\Omega(\mathcal{F})\cup(\operatorname{Sing}(\mathcal{F})\cap
\Omega(\mathcal{F}))\subset\Omega(\mathcal{F})\cup\operatorname{Sing}%
(\mathcal{F}).
\]
}

\textrm{The set $\Omega(\mathcal{F})$ is the union of all algebraic invariant
curves of $\mathcal{F}$. }
\end{notation}



\begin{proposition}
\label{Proposition:concentrationleafseparatrix} Let $L\subset{\mathcal{C }%
}_{\mu}(\Omega(\mathcal{F},T,k))$, then $\Gamma(L)$ is a stable graph of
$\mathcal{F}$.
\end{proposition}

Summing up, we have the following result.

\begin{proposition}
\label{Lemma:preparation} Let $\mathcal{F}$ be a foliation with non-degenerate
singularities in $\widetilde{\mathbb{C}\mathbb{P}^{m}}$. Then:

\begin{enumerate}
\item Each leaf $L\subset{\mathcal{C}}_{\mu}(\Omega(\mathcal{F},T,k))$ is algebraic;

\item Each leaf $L\in\Omega(\mathcal{F})$ accumulates at some singularity of
$\mathcal{F}$;

\item Let $L_{0}\subset{\mathcal{C}}_{\mu}(\Omega(\mathcal{F},T,k))$ and
$p\in\operatorname{Sing}(\mathcal{F})\cap\overline{L_{0}}$ be a non-$\mu
$-dicritical singularity and denote by $S(L_{0})_{p}$ the local separatrix
induced by $L_{0}$ at $p$. Then $p$ is a Siegel type singularity and
$S(L_{0})_{p}\subset{\mathcal{C}}_{\mu}(\Omega(\mathcal{F},T,k))$;

\item Let $L\subset{\mathcal{C}}_{\mu}(\Omega(\mathcal{F},T,k))$. Given a
singularity $p\in\operatorname{Sing}(\mathcal{F})\cap\overline{L}%
=\operatorname{Sing}(\mathcal{F})\cap\Lambda(L)$, for each separatrix
$S_{p}^{\prime}$ forming a Siegel pair with a local branch $S_{p}$ of
$\Lambda(L)$ at $p$, we also have $S_{p}^{\prime}\subset{\mathcal{C}}_{\mu
}(\Omega(\mathcal{F},T,k))$;

\item Given $L\subset{\mathcal{C}}_{\mu}(\Omega(\mathcal{F},T,k))$, the
corresponding Siegel component also satisfies $\Gamma(L)\subset{\mathcal{C}%
}_{\mu}(\Omega(\mathcal{F},T,k))$.
\end{enumerate}
\end{proposition}

\begin{proof}
Item (1) is proved in Lemma~\ref{Lemma:boundedorderalgebraic}. For item (2),
notice that the closure of $L\in\Omega(\mathcal{F})$ must be algebraic, thus
the result follows immediately. Item (3) is a straightforward consequence of
Lemma \ref{Lemma:analyticallylinearizablesing}. Item (4) comes directly from
Lemma \ref{Lemma:dulacaccumulation}, yielding immediately item (5).
\end{proof}


Let be given a leaf $L_{0}\subset{\mathcal{C}}_{\mu}(\Omega(\mathcal{F}%
,T,k))$. By Proposition~\ref{Lemma:preparation} above, $\Gamma(L_{0})$ is an
invariant set consisting of a union of algebraic curves containing
$\Lambda(L_{0})$ and such that $\Gamma(L_{0})\subset{\mathcal{C}}_{\mu}%
(\Omega(\mathcal{F},T,k))$.

\begin{lemma}
\label{Lemma:virtualholonomygamma} Given a leaf $L_{0}\subset{\mathcal{C}%
}_{\mu}(\Omega(\mathcal{F},T,k))$, all the virtual holonomy groups of
$\Gamma(L_{0})$ are finite and the normal singularities $q\in
\operatorname{Sing}(\mathcal{F})\cap\Gamma(L_{0})$ are all analytically
linearizable with rational eigenvalues. In case the Siegel component
$\Gamma(L_{0})$ of $L_{0}$ has just normal singularities, it is a stable graph
admitting a holomorphic first integral nearby it.
\end{lemma}

\begin{proof}
By hypothesis, all the singularities in $\Gamma(L_{0})$ are in the Siegel
domain, thus $\Gamma(L_{0})$ is a connected invariant subset contained in
$\Omega(\mathcal{F}))$. Applying
Lemma~\ref{Lemma:analyticallylinearizablesing} we conclude that each
singularity in $\Gamma(L_{0})$ is holonomy-analytically linearizable.
Furthermore, all the normal Siegel singularities are indeed linearizable. In
case $\Gamma(L_{0})$ has just normal singularities, then reasoning as in the
proof of Proposition \ref{Proposition:stablegraphstability} we can construct a
holomorphic first integral in a neighborhood of $\Gamma(L_{0})$.
\end{proof}

\begin{proof}
[End of the proof of Algebraic Stability theorem]Theorem~C is proved as
follows. By hypothesis $\operatorname*{med}(\Omega(\mathcal{F}))>0$, thus
Lemma \ref{Lemma:concentrationleaf} ensures the existence of a leaf
$L_{0}\subset{\mathcal{C}}_{\mu}(\Omega(\mathcal{F},T,k))\neq\emptyset$. By
Proposition \ref{Lemma:preparation}.(1) and Lemma
\ref{Lemma:virtualholonomygamma}, $L_{0}$ is algebraic and stable. Finally, in
the three dimensional case, if all the Siegel singularities are normal, then
Lemma~\ref{Lemma:virtualholonomygamma} assures that the Siegel component
$\Gamma(L_{0})$ of $L_{0}$ is a stable graph admitting a holomorphic first integral.
\end{proof}



\vglue.1in%

\begin{tabular}
[c]{ll}%
Leonardo Câmara & \quad Bruno Scárdua\\
Departamento de Matemática - CCE & \quad Instituto de Matemática\\
Universidade Federal do Espírito Santo & \quad Universidade Federal do Rio de
Janeiro\\
Av. Fernando Ferrari 514 & \quad Caixa Postal 68530\\
29075-910 - Vitória - ES & \quad21.945-970 Rio de Janeiro-RJ\\
BRAZIL & \quad BRAZIL\\
leonardo.camara@ufes.br & \quad scardua@im.ufrj.br
\end{tabular}


\begin{thebibliography}{99}                                                                                               %


\bibitem {Alexander-Verjovsky}J. C. Alexander \& A. Verjovsky, \emph{First
integrals for singular holomorphic foliations with leaves of bounded volume},
Lecture Notes in Mathematics, 1988, Volume 1345/1988, 1-10.

\bibitem {Arnold}V. I. Arnold, \emph{Geometrical methods in the theory of
ordinary differential equations}, Springer, 1983.

\bibitem {Brunella}M. Brunella, \emph{A global stability theorem for
transversely holomorphic foliations}, Ann. Global Anal. Geom. 15 (1997), no.
2, 179--186.

\bibitem {Brunellasiegel}M. Brunella, \emph{Inexistence of invariant measures
for generic rational differential equations in the complex domain}, Bol. Soc.
Mat. Mexicana (3), 2006.

\bibitem {Brunella-Nicolau}M. Brunella \& M. Nicolau, \emph{Sur les
hypersurfaces solutions des équations de Pfaff}\textit{,} C. R. Acad. Sci.
Paris Sér. I Math. 329 (1999), no. 9, 793--795.

\bibitem {Burnside}W. Burnside, \emph{On criteria for the finiteness of the
order of a group of linear substitutions}, Proc. London Math. Soc. (2) 3
(1905), 435-440.

\bibitem {CKP}C. Camacho, N. Kuiper \& J. Palis, \emph{The topology of
holomorphic flows with singularity}, Inst. Hautes Études Sci. Publ. Math., no.
48 (1978), 5--38.

\bibitem {C-LN}C. Camacho \& A. Lins Neto, \emph{Geometric theory of
foliations}\textit{,} Translated from the Portuguese by Sue E. Goodman.
Birkhäuser Boston, Inc., Boston, MA, 1985. vi + 205 pp.

\bibitem {Camacho-Scardua}{\ C. Camacho \& B. Scárdua}, \emph{Holomorphic
foliations with Liouvillian first integrals}, {Ergodic Theory and Dynamical
Systems} (2001), 21, pp.717-756.

\bibitem {scarduaJMAA}{\ C. Camacho \& B. Scárdua}, \emph{On the existence of stable 
compact leaves for transversely holomorphic foliations}, pre-print 2011,
submitted, arXiv:1204.0095v1 [math.GT].




\bibitem {Camara-Scardua-separatrix}L. M. Câmara \& B. Scardua, \emph{A survey
on the local type of Siegel singularities}, in preparation.

\bibitem {Carnicer}M. Carnicer, \emph{The Poincaré problem in the
non-dicritical case}, Ann. of Math. 140 (1994) 289-294.

\bibitem {C-LN-S2}C. Camacho, A. Lins Neto \& P. Sad, \emph{Foliations with
algebraic limit sets}, Ann. of Math. 136 (1992), 429--446.

\bibitem {Correa-Soares}M. Corrêa Jr. \& M. G. Soares, \emph{A note on
Poincaré's problem for quasi-homogeneous foliations}, Proc. Amer. Math. Soc.
140 (2012), no. 9, 3145-3150.

\bibitem {Dulac}H. Dulac, \emph{Solutions d'un système de équations
différentiale dans le voisinage des valeus singulières}, Bull. Soc. Math.
France 40 (1912) 324--383.

\bibitem {Elizarov}P. M. Elizarov \& Yu. S. Ilyashenko, \emph{Remarks on the
orbital analytic classification of germs of vector fields}\textit{,} Math.
USSR Sb. \textbf{49} (1984) 111-124.

\bibitem {Godbillon}C. Godbillon, \emph{Feuilletages: Études géométriques},
Progress in Mathematics, 98. Birkhäuser Verlag, Basel, 1991. xiv + 474 pp.

\bibitem {Ghys}E. Ghys, \emph{À propos d'un théorème de J.-P. Jouanolou
concernant les feuilles fermées des feuilletages holomorphes}, Rend. Circ.
Mat. Palermo (2) 49 (2000), no. 1, 175--180.

\bibitem {GM}X. Gomez-Mont, \emph{Integrals for holomorphic foliations with
singularities having all leaves compact}, Annales de l'institut Fourier (1989)
Volume: 39, Issue: 2, Publisher: Institut Fourier, page 451-458.

\bibitem {Grobman 1959}D.M. Grobman, \emph{Homeomorphism of systems of
differential equations}. Dokl. Akad. Nauk. USSR, 128, 880--881 (1959).

\bibitem {Grobman 1962}D. M. Grobman, \emph{Topological classification of
neighbourhoods of a singularity in }$n$\emph{-space}, Math. Sbornik, 56,
77--94 (1962).

\bibitem {Gunning-Rossi}R. C. Gunning \& H. Rossi, \emph{Analytic Functions of
Several Complex Variables}, Prentice Hall; Englewood Cliffs; N.J. 1965.

\bibitem {Hartshorne}R. Hartshorne, \emph{Algebraic Geometry}, Springer 1977.

\bibitem {Helena}H. Reis, \emph{Equivalence and semi-completude of
foliations}, Nonlinear Analysis 64 (2006) 1654-1665.

\bibitem {Jouanolou}J.-P. Jouanolou, \emph{Équations de Pfaff algèbriques},
Lecture Notes in Math. 708, Springer-Verlag, Berlin, 1979.

\bibitem {LinsNeto}A. Lins Neto, \emph{Algebraic solutions of polynomial
differential equations and foliations in dimension two}, Lect. Notes in Math.
no. 1345, 192--231.

\bibitem {LinsNeto-Soares}A. Lins Neto \& M. G. Soares, \emph{Algebraic
solutions of one-dimensional foliations}, J. Differential Geom. 43 (1996), no.
3, 652-673.

\bibitem {Mattei-Moussu}J.-F. Mattei \& R. Moussu, \emph{Holonomie et
intégrales premières}, Ann. Sci. École Norm. Sup. (4) \textbf{13} (1980), 469--523.

\bibitem {Santos-Scardua}F. Santos \& B. Scardua, \emph{Stability of complex
foliations transverse to fibrations}, to appear in Proceedings of the American
Mathematical Society.

\bibitem {Reeb1}G. Reeb, \emph{Variétés feuilletées, feuilles voisines},
C.R.A.S. Paris 224 (1947), 1613-1614.

\bibitem {Schur}I. Schur, \emph{Über gruppen periodischer substitutionen},
Sitzungsber. Preuss. Akad. Wiss. (1911), 619--627.

\bibitem {scarduaIndagationes}B. Scárdua, \emph{Complex projective foliations
having subexponential growth}, Indagationes Math. N.S. 12 (3) pp. 293-302
Sept. 2001.

\bibitem {scarduaJDCS}{\ B. Azevedo Scárdua}, \emph{Integration of complex
differential equations}, {\ Journal of Dynamical and Control Systems}, issue
1, vol. \textbf{5}, pp.1-50, 1999.

\bibitem {seidenberg}A. Seindenberg,\textit{\ }\emph{Reduction of the
singularities of the differentiable equation }$Ady=Bdx$, Amer. J. Math
\textbf{90} (1968), 248-269.

\bibitem {Soares1}M. G. Soares, \emph{On the geometry of Poincaré's problem
for one-dimensional projective foliations}, An. Acad. Brasil. Ciênc. 73
(2001), no. 4, 475-482.

\bibitem {Soares2}M. G. Soares, \emph{Projective varieties invariant by
one-dimensional foliations}, Ann. of Math. (2) 152 (2000), no. 2, 369-382.

\bibitem {Soares3}M. G. Soares, \emph{The Poincaré problem for hypersurfaces
invariant by one-dimensional foliations}\textit{.} Invent. Math. 128 (1997),
no. 3, 495-500. (Reviewer: Gary P. Kennedy) 32L30 (14M10 57R30)

\bibitem {Soares5}M. G. Soares, \emph{Invariant algebraic sets of
foliations}\textit{,} Singularity theory (Trieste, 1991), 800-816, World Sci.
Publ., River Edge, NJ, 1995.

\bibitem {Soares6}M. G. Soares, \emph{A note on algebraic solutions of
foliations in dimension two}\textit{,} Dynamical systems (Santiago, 1990),
250-254, Pitman Res. Notes Math. Ser., 285, Longman Sci. Tech., Harlow, 1993.

\bibitem {IlyYak2006}Yu. Ilyashenko \& S. Yakovenko, \emph{Lectures on
analytic differential equations}\textit{, }Graduate Studies in Mathematics,
American Mathematical Society (December 27, 2007).
\end{thebibliography}
\end{document}